\documentclass[12pt]{amsart}
\usepackage[T1]{fontenc}
\usepackage[latin1]{inputenc}
\usepackage{fullpage, appendix}
\usepackage{verbatim} 
\usepackage{amsmath}
\usepackage{amssymb}
\usepackage{float}

\usepackage{color}
\usepackage{bbding}
\usepackage{wrapfig}
\usepackage{pst-node}
\usepackage[sorting=nyt]{biblatex}

\usepackage{lineno}
\makeatletter
\@addtoreset{equation}{section}
\makeatother

\usepackage{hyperref}
\usepackage{cleveref}

\newtheorem{theorem}{Theorem}[section]  
\newtheorem{lemma}{Lemma}[section] 

\newtheorem{definition}{Definition}[section]
\newtheorem{remark}{Remark}[section]

\newtheorem{conjecture}{Conjecture}[section]

\begin{document}


\title{Volume comparison for symmetric spaces of non-compact type of rank 1}

\author{Jiaqi Chen}
\address{School of Electrical Engineering and Automation, Xiamen University of Technology, Xiamen 361024, Fujian, P.R. China.}
\email{chenjiaqi@xmut.edu.cn}

\author{Yufei Shan}
\address{School of Mathematical Sciences, Shanghai Jiao Tong University,  Shanghai 200240, P.R. China.}
\email{universeplane1991@sjtu.edu.cn}

\author{Yinghui Ye}
\address{Department of Mathematics, Sun Yat-sen University, Guangzhou, Guangdong 510275, P.R. China.}
\email{yeyh8@mail2.sysu.edu.cn}

\subjclass[2020]{53B20; 53E20}

\thanks{Yinghui Ye is partially supported by National Key R\&D Program of China (2024YFA1015000) and National Natural Science Foundation of China (No.12571065); Jiaqi Chen is partially supported by the Scientific Research Foundation of Xiamen University of Technology under Grant (No.YKJ23009R), the Scientific Research Foundation for Young and middle-aged Teachers in Fujian Province under Grant (No.JAT231105); Yufei Shan is partially supported by National Natural Science Foundation of China (No.12271348)}

\begin{abstract}
    Motivated by Schoen's conjecture on the volume functional for closed hyperbolic manifolds, we generalize the volume comparison theorem of Hu, Ji, and Shi and establish a volume comparison theorem for rank 1 symmetric spaces of non-compact type under a scalar curvature condition. Furthermore, we prove a rigidity result. Our proof uses the normalized Ricci--DeTurck flow to analyze the asymptotic behavior of the volume functional and to derive monotonicity properties. This extends the classical volume comparison framework to symmetric spaces of non-compact type.
\end{abstract}
\maketitle
\section{Introduction}

The volume of a Riemannian manifold is a fundamental geometric quantity that plays a central role in geometric analysis. The classical Bishop-Gromov volume comparison theorem assumes a lower bound on the Ricci curvature. Schoen proposed the following conjecture concerning the volume functional $\mathrm{Vol}(\cdot)$ of a closed hyperbolic manifold.
\begin{conjecture}[Schoen's Conjecture~\cite{Sc1989}]
    Let $(M^n, g_0)$ be a closed hyperbolic manifold. Let $g$ be another metric on $M$ with scalar curvature $R(g) \geq R(g_0)$, then $\mathrm{Vol}_g(M) \geq \mathrm{Vol}_{g_0} (M)$.
\end{conjecture} 

This conjecture remained largely open until the three-dimensional case was resolved as a corollary of the work of Hamilton and Perelman on the geometrization conjecture \cite{Pe2002, Pe2003}. Subsequently, Agol, Storm, and Thurston established similar results for compact hyperbolic 3-manifolds with minimal surface boundaries \cite{AST2007}. In \cite{MT2009}, Miao and Tam demonstrated that the conjecture for closed manifolds does not directly extend to manifolds with boundary under only Dirichlet boundary conditions. Specifically, they provided a counterexample involving geodesic balls in three-dimensional hyperbolic space, using variations of the volume functional under scalar curvature constraints. In \cite{Yu2023}, Yuan proved the volume comparison theorem for geodesic balls in V-Static spaces under some boundary conditions. Subsequently, Lin and Yuan explored the volume comparison theorem concerning the Q-curvature \cite{LY2016, LY2022}, and Chen, Fang, He and Zhong generalized the result of Lin and Yuan with respect to $\sigma_k$-curvature in \cite{CFHZ2025}.

A natural question is whether an analogous conjecture holds for non-compact Einstein manifolds. In \cite{HJS2016}, Hu, Ji, and Shi gave an affirmative answer for perturbations of conformally compact Einstein manifolds. In this paper, we generalize their result to rank 1 symmetric spaces of non-compact type. 

In order to state our main result, we need to introduce some notation and basic concepts. Let $(M^n, g_0)$ be a Riemannian manifold. We denote the global norm by $\| \cdot \|_{C^{k}}$, the pointwise norm by $| \cdot |$, and the distance function between $x, p_0\in M$ by $\mathrm{dist}(x, p_0)$ (or simply $d(x, p_0)$ or $d(x)$). We denote by $\nabla$ the covariant derivative, by $R_{ijkl}$ the Riemann curvature tensor with indices $i, j, k, l$ (or simply $\mathrm{Rm}$), by $R_{ij}$ the Ricci curvature tensor with indices $i, j$, by $B(x, r)$ the geodesic ball centered at $x\in M$ with radius $r$, and by $d\mu_{g_0}$ the volume form. All these are with respect to the background metric $g_0$, unless otherwise stated. Moreover, we use $g_{0}^{ij}$ for the inverse of the metric $g_0$ and $\hat{g}^{ij}$ for the inverse of another metric $\hat{g}$.

We call an isometry $\Phi : M \rightarrow M$ with $\Phi(p)=p$, a reflection at $p\in M$, if $d \Phi|_{p}=-\mathrm{id}_{T_{p}M}$. A Riemannian manifold $(M^n, g_0)$ is called a \textbf{symmetric space} if for an arbitrary point $p\in M$ there exists a reflection $\Phi_p$. One of the most important properties of symmetric spaces is that the Riemann curvature tensors of symmetric spaces are parallel, i.e. $\nabla \mathrm{Rm} \equiv 0$. 

By the de-Rham Decomposition Theorem, the universal covering of the symmetric spaces can be expressed as a product $M_1 \times \cdots \times M_m$ of irreducible symmetric spaces. All $M_i$ are Einstein. If all Einstein constants $\lambda_i$ are negative, then, $M$ is said to be \textbf{of non-compact type}. Let $g$ be another arbitrary metric on $M$. As defined in \cite{HJS2016}, we introduce the concept of \textbf{relative volume} of $g$ with respect to $g_0$ by 
\begin{align}\label{definition of relative volume}
    V_{g_0}(g) = \lim_{i\rightarrow +\infty} \left(\mathrm{Vol}(\Omega_i, g) - \mathrm{Vol}(\Omega_i, g_0) \right), 
\end{align}
 where $\{\Omega_i\}_{i = 1}^{+\infty}$ is a compact exhaustion of $M$, i.e. $\Omega_i \subset M$ are compact sets satisfying $\lim_{i \rightarrow +\infty} \Omega_i = M$. We shall see that if $g$ has sufficient decay relative to $g_0$, i.e. $\|e^{\tau d(x)}(g - g_0)\|_{C^0} = O(1)$, where $d(x)$ is the distance function with respect to $g_0$ between $x\in M$ and a fixed point $p_0 \in M$ and $\tau > \overline{\lim}_{d \rightarrow +\infty} \Delta_{g_0} d$ is a constant, then, the relative volume of $g$ with respect to $g_0$ exists and is independent of the choice of the compact exhaustion (\Cref{existence of relative volume}). Following the idea of \cite{HJS2016}, we will utilize the normalized Ricci-DeTurck flow \eqref{NRDF} to prove the following "volume" comparison theorem for symmetric spaces of non-compact type of rank 1. 
\begin{theorem}\label{main theorem}
    Suppose that $(M^n, g_0)$ is a simply connected irreducible symmetric space of non-compact type of rank 1 with $n \geq N$, where $N = 4, 6, 4, 16$ for real, complex, quaternionic and octonionic hyperbolic spaces respectively. Let $g$ be another Riemannian metric on $M$. Then, for $\tau > \lim_{d\rightarrow+\infty}\Delta_{g_0} d(x)$, there exists some $\varepsilon_0 > 0$ such that if 
    \begin{align*}
        \|e^{\tau d}(g - g_0)\|_{C^{1}} \leq \varepsilon_0,
    \end{align*}
    and 
    \begin{align*}
        R \geq R_0, 
    \end{align*}
    where $R$ and $R_0$ are the scalar curvatures with respect to $g$ and $g_0$ respectively, then, 
    \begin{align*}
        V_{g_0}(g) \geq 0. 
    \end{align*}
    In particular, if $V_{g_0}(g) = 0$, there exists a $C^{\infty}$ diffeomorphism $\Phi: M \rightarrow M$, such that $g = \Phi^{*}g_0$. 
\end{theorem}
\begin{remark}
    If $M = \mathbb{H}^n$, then, \Cref{main theorem} is Theorem 1.7 in \cite{HJS2016}.
\end{remark}

This paper is organized as follows. In \Cref{sec: symmetric_space}, we review the foundational geometric framework of symmetric spaces including the definition of irreducible symmetric spaces of non-compact type, their root systems and key decomposition theorems. Additionally, we derive explicit expressions for curvature and Laplacian operators along geodesics, which are critical for later analysis. In \Cref{sec: ricci_flow}, we introduce the normalized Ricci flow \eqref{NRF} and the normalized Ricci-DeTurck flow \eqref{NRDF} and review the short-time existence results and stability results under some perturbation conditions. In \Cref{sec: longtime}, following the approach of \cite{HJS2016}, we analyze the asymptotic properties of the normalized Ricci-DeTurck flow on Cartan-Hadamard manifolds which include symmetric spaces of non-compact type. In \Cref{sec: relative_volume}, we derive the evolution equation for relative volume under the normalized Ricci-DeTurck flow \eqref{NRDF}, prove its monotonicity, and use this to establish non-negativity of the relative volume and the rigidity result in \Cref{main theorem}.

\section{Symmetric Spaces of Non-compact Type}\label{sec: symmetric_space}

In this section, we briefly review basic facts about symmetric spaces of non-compact type. We then express two key quantities 
\begin{align}\label{two important quantities}
    \gamma := \sup_{x \in M} \sup_{h \in T_x^*M \otimes T_x^*M \setminus \{0\}} \frac{R_{ijkl}(x) h^{ik} h^{jl}}{|h|^2}\;\;\; \text{and}\;\;\; \tau_0 := \lim_{d\rightarrow +\infty} \Delta_{g_0} d,
\end{align}
where $d(x) := \text{dist}(x, p_0)$ and $R_{ijkl}$ are components of the Riemann curvature tensor of $g_0$, in terms of the associated Lie algebraic structures of symmetric spaces (\Cref{laplacian distance function,curvature representation lemma}). And using the representation theory of semisimple Lie algebras, we are able to figure out the exact values of these two quantities. These constants will play a crucial role in the analysis of the normalized Ricci-DeTurck flow on symmetric spaces of non-compact type, which is carried out in the next section.

\subsection{Definitions and Basic Properties}
Let $(M^n, g_0)$ be a simply connected irreducible symmetric space of non-compact type and $p_0\in M$ be any fixed point in $M$. Let $G$ be the identity connected component of the isometry group of $M$ with its Lie algebra denoted as $\mathfrak{g}$ and $K$ be the subgroup of $G$ which leaves $p_0$ fixed with its Lie algebra denoted as $\mathfrak{l}$. Then, $G$ acts on $M$ transitively, which implies that $M\cong G / K$. The reflection of $M$ at $p_0$, denoted as $\Phi_{p_0}$, can induce an involution $\xi$ (i.e. $\xi\circ\xi = id$) on $G$ by  
\begin{equation}
    \begin{split}
        \xi: G &\rightarrow G, \\
             g &\mapsto \Phi_{p_0}g \Phi_{p_0}.
    \end{split}
\end{equation}
And the tangent map of this $\xi$ at the identity, denoted by $d\xi : \mathfrak{g} \rightarrow \mathfrak{g}$, is an involution as well. Hence, $\mathfrak{g}$ admits the eigenspace decomposition  
\begin{equation}
    \mathfrak{g} = \mathfrak{p}\oplus\mathfrak{l},
\end{equation}
where $\mathfrak{p}$ and $\mathfrak{l}$ are the eigenspaces for the eigenvalues -1 and 1 of the $d\xi$ respectively. 
\begin{remark}
    $\mathfrak{l}$ here agrees with the Lie algebra of $K$. Thus, we just use the same notation to denote them. In addition, throughout this paper, the Lie algebra of a Lie group is identified with the space of left-invariant vector fields of that Lie group, equipped with the corresponding Lie bracket on vector fields.
\end{remark}
One can easily verify that $[\mathfrak{p}, \mathfrak{p}] \subset \mathfrak{l}$, $[\mathfrak{p}, \mathfrak{l}] \subset \mathfrak{p}$ and $[\mathfrak{l}, \mathfrak{l}] \subset \mathfrak{l}$. 
Define 
\begin{equation}
    \begin{split}
        \pi: G &\rightarrow M, \\
             h &\mapsto h(p_0),
    \end{split}
\end{equation}
and denote $d\pi_e$ as the tangent map of $\pi$ at the identity element of $G$. By the transitivity of the action of $G$ on $M$, $\pi$ is surjective and so is $d\pi_{e}$. Moreover, it is clear that $\mathrm{ker}(d\pi_{e}) = \mathfrak{l}$, which implies that $\mathfrak{p} \cong T_{p_0} M$. In addition, there is a canonical Lie algebra isomorphism between the Lie algebra $\mathfrak{g}$ and the space of Killing field $K(M)$, 
\begin{equation}
    \begin{split}
        \sigma: \mathfrak{g} &\rightarrow K(M), \\
                           x &\mapsto \hat{X},
    \end{split}
\end{equation}
where $\hat{X}(p) := \frac{d}{dt} \exp(x t) (p)|_{t = 0}$ for any $p\in M$, and $\exp(x t)\in G$ is the exponential map for the Lie group $G$ with the parameter $t\in \mathbb{R}$. It is worth noting that  
\begin{align}\label{relation between lie algebra and kill field}
    \sigma([x, y]) = -[\sigma(x), \sigma(y)], 
\end{align}
since the Lie bracket on $\mathfrak{g}$ is defined as the Lie bracket for the corresponding left invariant vector fields, and the curve $\gamma(t): =\exp(xt)(p)$ is a geodesic starting at $p$, in the direction $\sigma(x)|_{p}$.

On the Lie algebra $\mathfrak{g}$, we denote the Killing form as $\langle \cdot , \cdot \rangle$ which is defined as following 
\begin{equation}
    \langle x , y \rangle := \mathrm{Tr} \left( ad(x) \circ ad(y) \right),
\end{equation}
where 
\begin{equation}
    \begin{split}
        ad(x): \mathfrak{g} &\rightarrow \mathfrak{g}, \\
               z &\mapsto [x, z],
    \end{split}
\end{equation}
for $x, y, z\in \mathfrak{g}$. The Killing form satisfies 
\begin{equation}\label{compactibility between killing form and lie brack}
    \langle [x,y], z\rangle = -\langle y, [x,z]\rangle, 
\end{equation}
which implies $\langle \mathfrak{p}, \mathfrak{l}\rangle = 0$. Moreover, it is positive definite on $\mathfrak{p}$ and negative definite on $\mathfrak{l}$. Since $(M, g_0)$ is an irreducible symmetric space, then, by Schur's lemma, there exists a constant $\lambda>0$ such that
\begin{align}\label{metric = Killing form}
g_0\bigl(\sigma(x), \sigma(y)\bigr)\big|_{p_0} = \lambda \langle x, y\rangle,
\end{align}
for all $x,y \in \mathfrak{p}$. For different choices of $\lambda>0$, the associated covariant derivatives coincide. We therefore fix a convenient value of $\lambda$ for computational simplicity in what follows. Moreover, the covariant derivatives and curvature tensors of a symmetric space can be expressed in terms of the underlying Lie algebraic structure,
\begin{equation}\label{sectional curvature}
    \begin{split}
        (\nabla_{\hat{X}} Y)|_{p_0} =& [\hat{X}, Y]|_{p_0},\\
        R(\hat{X}, \hat{Y}) \hat{Z}|_{p_0} =& -[[\hat{X}, \hat{Y}], \hat{Z}]|_{p_0} = -\sigma([[x, y], z])|_{p_0},
    \end{split}
\end{equation}
where $\hat{X}$, $\hat{Y}$, $\hat{Z}$ are equal to $\sigma(x),\sigma(y),\sigma(z)$ respectively, for any $x, y, z \in \mathfrak{p}$, $Y$ is an arbitrary vector field on $M$, and $R(\hat{X}, \hat{Y})\hat{Z} := -\nabla_{\hat{X}}\nabla_{\hat{Y}}\hat{Z} + \nabla_{\hat{Y}}\nabla_{\hat{X}}\hat{Z} + \nabla_{[\hat{X}, \hat{Y}]}\hat{Z}$ (refer to Theorems~3.3 and~4.2 in~\cite{He1978} or Section 3.4 in \cite{Ba2015} for details).

\subsection{Geometric Structures}
Next, we introduce the root system, which allows us to compute the two quantities defined in \eqref{two important quantities}. Readers may refer to \cite{He1978} for more details.

\(\mathfrak{a} (\subset \mathfrak{p})\) is called a \textbf{maximal abelian subalgebra} if it is not contained in any larger abelian subalgebra of \(\mathfrak{p}\). The dimension of $\mathfrak{a}$ is called the \textbf{rank of the corresponding symmetric space} (denoted by $r$). Let $\mathfrak{a}^*$ be the dual space of $\mathfrak{a}$. We say that $a\in\mathfrak{a}^*$ is a \textbf{root} of $\mathfrak{g}$ relative to $\mathfrak{a}$ if \(a \neq 0\) and there exists \(x \neq 0 \in \mathfrak{g}\) such that \([v, x]=a(v) x\) for all $v\in\mathfrak{a}$. Denote by $\Delta$ the set of all roots, which is called the \textbf{root system}. Denote by $\mathfrak{g}_{a}$ the corresponding eigenspace of $a$. Let $v_{0} \in \mathfrak{a}$ be such that $a(v_{0})\neq0$ for all nonzero $a\in\Delta$, and define the set of \textbf{positive roots} by $\Delta_{+}=\{a\in\Delta: a(v_{0})>0\}$.

\begin{remark}
    The root system considered here is not the usual root system of a complex semisimple Lie algebra. It is sometimes called the \emph{real} root system associated with $\mathfrak{p}$; it corresponds to symmetric spaces of non-compact type and can be classified by Satake diagrams (see Chapter~4.3 in~\cite{OV1994}).
\end{remark}

The existence of the involution $d\xi$ implies $-\Delta=\Delta$ and $d\xi$ maps $\mathfrak{g}_{a}$ to $\mathfrak{g}_{-a}$. Therefore, if we set  $\mathfrak{p}_{a} := (\mathfrak{g}_{a} \oplus \mathfrak{g}_{-a}) \cap \mathfrak{p}$ and $\mathfrak{l}_{a} := (\mathfrak{g}_{a} \oplus \mathfrak{g}_{-a}) \cap \mathfrak{l}$.  Then, we have the following root space decomposition
\begin{align*}
	\mathfrak{g}=
    &\mathfrak{a}\oplus_{a\in\Delta_{+}}(\mathfrak{g}_{a}\oplus\mathfrak{g}_{-a})\oplus\mathfrak{l}_{0} \\
	=&\mathfrak{p}\oplus\mathfrak{l}\\
    =&\mathfrak{a}\oplus_{a\in\Delta_{+}}(\mathfrak{p}_{a}\oplus\mathfrak{l}_{a})\oplus\mathfrak{l}_{0}. 
\end{align*}
These splittings are orthogonal with respect to the Killing form. In addition, the subspace $\mathfrak{l}_{0}$ is a subalgebra. Using the Jacobi identity, we conclude that for any $a,b\in\Delta$, $[\mathfrak{g}_{a},\mathfrak{g}_{b}]\subset\mathfrak{g}_{a+b}$. Hence, $\mathfrak{n}_{+}:=\oplus_{a\in\Delta_{+}}\mathfrak{g}_{a}$ and $\mathfrak{n}_{-}:=\oplus_{a\in\Delta_{+}}\mathfrak{g}_{-a}$ are nilpotent Lie algebras, and $d\xi(\mathfrak{n}_{+})=\mathfrak{n}_{-}$. Define  
\begin{align}\label{inner product on lie algebra}
	(\cdot, \cdot): = -\lambda \langle \cdot, d\xi(\cdot) \rangle,
\end{align}
which is an inner product on $\mathfrak{g}$. Moreover, by \eqref{metric = Killing form}, it is straightforward to verify that
\begin{align*}
    g_0(\sigma(x), \sigma(y))|_{p_0} = (x, y),\quad \text{for any }\; x, y\in \mathfrak{p}. 
\end{align*}

Next, we will introduce an orthonormal frame on a symmetric space of non-compact type, on which we will carry out the computation for the two quantities defined in \eqref{two important quantities}. Let $a_{1},\cdots, a_{n-r}$ be the roots of $\Delta_{+}$ occurring with the appropriate multiplicities and $x_{1},\cdots,x_{n-r}$ be an orthonormal basis of $\mathfrak{n}_{+}$ with respect to $(\cdot, \cdot)$ such that $x_{i} \in \mathfrak{g}_{a_{i}}$ and let $y_i := d\xi(x_i)$. Then, 
\begin{equation}\label{basic innter product relation}
\lambda\left\langle x_i, y_j\right\rangle=-\delta_{i j}, \quad\left\langle x_i, x_j\right\rangle=0,\quad\text{ for  }\; 1\leq i, j \leq n -r . 
\end{equation}
We set
\begin{equation}
    p_{i}=\frac{1}{\sqrt{2}}(x_{i}-y_{i}),\quad k_{i}=\frac{1}{\sqrt{2}}(x_{i} +y_{i}). 
\end{equation}
By \eqref{basic innter product relation}, it is straightforward to check that $\{p_i\}_{i=1}^{n-r}$ and $\{k_i\}_{i=1}^{n-r}$ form orthonormal bases of $\mathfrak{a}^{\perp}\subset\mathfrak{p}$ and $\mathfrak{l}_0^{\perp}\subset\mathfrak{l}$, respectively, with respect to $(\cdot,\cdot)$. Moreover, for any $v\in \mathfrak{a}$,
\begin{equation}\label{basic frame relation}
    \begin{split}
        [p_i,v] =& \frac{1}{\sqrt{2}}([x_i,v]-[y_i,v]) = -a_i(v) k_i,\\
        [k_i,v] =& \frac{1}{\sqrt{2}}([x_i,v]+[y_i,v]) = -a_i(v) p_i,
    \end{split}
\end{equation}
and 
\begin{equation}\label{sinh existence}
    \begin{split}
        \exp(ad(v))(p_i) =& \mathrm{sh}(a_i(v))k_i + \mathrm{ch}(a_i(v)) p_i,\\
        \exp(ad(v))(k_i) =& \mathrm{sh}(a_i(v))p_i + \mathrm{ch}(a_i(v)) k_i,
    \end{split}
\end{equation}
where $ad(v)(x) := [v, x]$ and 
\begin{equation}\label{adjoint action}
    \exp(ad(v))(x) := \sum_{n = 0}^{+\infty}\frac{1}{n!}(ad(v))^n(x),
\end{equation}
for any $x\in \mathfrak{g}$. Note that 
\begin{equation*}
    d\xi([x_{i}, y_{i}]) = [y_{i}, x_{i}] = -[x_{i}, y_{i}],
\end{equation*}
so $[x_i,y_i]\in\mathfrak{p}$. Moreover, for any $v\in\mathfrak{a}$,
\begin{align*}
    [[x_i, y_i], v] = [[x_i, v], y_i] + [x_i, [y_i,v]] = (a_i(v) - a_i(v))[x_i, y_i] = 0, 
\end{align*}
which implies that $[x_{i},y_{i}]\in\mathfrak{a}$ (since $\mathfrak{a}$ is a maximal abelian subalgebra in $\mathfrak{p}$). Furthermore, by \eqref{compactibility between killing form and lie brack}, \eqref{inner product on lie algebra} and \eqref{basic innter product relation}, for any $v \in \mathfrak{a}$, 
\begin{equation}
	([x_i, y_i], v) = -\lambda\langle [x_{i}, y_{i}], d\xi(v) \rangle = -\lambda\langle [x_{i}, v], y_{i} \rangle = a_{i}(v) \lambda\langle x_{i}, y_{i} \rangle = -a_{i}(v), 
\end{equation}
and hence
\begin{equation}
	[x_{i}, y_{i}] = -a_{i}^{\#} = -\sum_{k = 1}^{r} a_i (v_k) v_k, 
\end{equation}
and 
\begin{equation}\label{basic frame relation 1}
    [p_i, k_i] = \frac{1}{2}[x_i - y_i, x_i + y_i] = [x_i,y_i] = -\sum_{k = 1}^{r}a_i(v_k)v_k,
\end{equation}
where $v_{1},\dots,v_{r}$ is an orthonormal basis of $\mathfrak{a}$ with respect to $(\cdot,\cdot)$. Now, $v_1, \cdots, v_{r}$, $p_1,\cdots, p_{n - r}$, and $k_1, \cdots, k_{n - r}$ defined above provide a convenient frame for $(M, g_0)$, with which we can write the corresponding covariant derivatives along the geodesic ray $\exp(vt)(p_0)$ for any $t > 0$ and $v \in \mathfrak{a}$. 

\begin{lemma}\label{spherical frame}
    Let $(M^n, g_0)$ be a simply connected irreducible symmetric space of non-compact type of rank $r$. Then, under the above notation, for any $v\in \mathfrak{a}$, we have that 
    \begin{equation}\label{orthonormal basis on radius}
        \begin{split}
            &g_0(\sigma(v_i), \sigma(k_j))|_{\exp(v t)(p_0)} = 0,\\
            &g_0(\sigma(v_{i_1}),\sigma(v_{i_2}))|_{\exp(vt)(p_0)} = \delta_{i_1i_2},\\
            &g_0(\sigma(k_{j_1}), \sigma(k_{j_2}))|_{\exp(vt)(p_0)} = \mathrm{sh}^2(a_{j_1}(v)t)\delta_{j_1j_2},
        \end{split}
    \end{equation}
    for any $1\leq i, i_1, i_2 \leq r$ and $1 \leq j, j_1, j_2 \leq n - r$. Moreover, we have that 
    \begin{itemize}
        \item for $1 \leq i \leq r$,
        \begin{equation}\label{covariant derivative on radius 1}
            \nabla_{\sigma(v_i)}\sigma(v_i)|_{\exp(vt)(p_0)} = 0,
        \end{equation}
        \item for $1 \leq j \leq n - r$,
        \begin{equation}\label{covariant derivative on radius 2}
            \nabla_{\sigma(k_j)}\sigma(k_j)|_{\exp(vt)(p_0)}\\ 
            = \mathrm{sh}(-a_j(v)t)\cdot \mathrm{ch}(-a_j(v)t)\cdot \sum_{l = 1}^{r}a_j(v_l)\sigma(v_l)|_{\exp(vt)(p_0)}. 
        \end{equation}
    \end{itemize}
\end{lemma}
\begin{proof}
    For \eqref{orthonormal basis on radius}, since $\exp(-v t)$ is in the isometry group of $M$ for any $v\in \mathfrak{a}$, it suffices to show that 
    \begin{align*}
        &g_0(\exp(-vt)_*(\sigma(v_i)|_{\exp(vt)(p_0)}),\exp(-vt)_*(\sigma(k_j)|_{\exp(vt)(p_0)})) = 0,\\
        &g_0(\exp(-vt)_*(\sigma(v_{i_1})|_{\exp(vt)(p_0)}),\exp(-vt)_*(\sigma(v_{i_2})|_{\exp(vt)(p_0)})) = \delta_{i_1i_2},\\
        &g_0(\exp(-vt)_*(\sigma(k_{j_1})|_{\exp(vt)(p_0)}),\exp(-vt)_*(\sigma(k_{j_2})|_{\exp(vt)(p_0)})) = \mathrm{sh}^2(a_{j_1}(v)t)\delta_{j_1j_2}, 
    \end{align*}
    for $1 \leq i, i_1, i_2 \leq r$ and $1 \leq j, j_1, j_2 \leq n - r$. In fact, for $1 \leq i \leq r$, 
    \begin{align*}
        \exp(-vt)_*(\sigma(v_i)|_{\exp(vt)(p_0)}) 
        =& \exp(-vt)_*(\frac{d}{ds} \bigg|_{s = 0} \exp(v_is)\cdot\exp(vt)(p_0))\\
        =& \frac{d}{ds} \bigg|_{s = 0} \exp(-vt)\cdot\exp(v_is)\cdot \exp(vt)(p_0)\\
        =& \frac{d}{ds} \bigg|_{s = 0} \exp([\exp(ad(-vt))v_i]s)\\
        =& \sigma(\exp(ad(-vt))v_i)|_{p_0}\\
        =& \sigma(v_i)|_{p_0},
    \end{align*}
    and for $1\leq j \leq n - r$, by \eqref{sinh existence}, we have that 
    \begin{align*}
        \exp(-vt)_*(\sigma(k_j)|_{\exp(vt)(p_0)}) 
        =& \exp(-vt)_*(\frac{d}{ds}  \bigg|_{s = 0}  \exp(k_js) \cdot \exp(vt)(p_0))\\
        =& \frac{d}{ds} \bigg|_{s = 0} \exp(-vt)\cdot\exp(k_js)\cdot\exp(vt)(p_0)\\
        =& \frac{d}{ds} \bigg|_{s = 0} \exp([\exp(ad(-vt))(k_i)]s)(p_0)\\
        =& \frac{d}{ds} \bigg|_{s = 0} \exp([\mathrm{sh}(-a_i(v)t)p_i + \mathrm{ch}(-a_i(v)t)k_i]s)(p_0)\\
        =& \sigma(\mathrm{sh}(-a_i(v)t)p_i + \mathrm{ch}(-a_i(v)t)k_i)|_{p_0}\\
        =& \mathrm{sh}(-a_i(v)t)\sigma(p_i)|_{p_0}. 
    \end{align*}
    Then, together with \eqref{metric = Killing form} and \eqref{inner product on lie algebra}, \eqref{orthonormal basis on radius}) follows. 

    For \eqref{covariant derivative on radius 1} and \eqref{covariant derivative on radius 2}, let $1 \leq j \leq n - r$. The integral curve $\gamma_j(s)$ of the Killing field $\sigma(k_j)$ passing through $\exp(vt)(p_0)$ is $\gamma_j(s) := \exp(k_js)\cdot \exp(vt)(p_0)$. Then, $\sigma(k_j)|_{\gamma_j(s)} = \dot\gamma_j(s)$. Thus, 
    \begin{align*}
        \nabla_{\sigma(k_j)}\sigma(k_j)|_{\exp(vt)(p_0)} = \nabla_{\dot\gamma_j(s)}\dot\gamma_j(s)|_{s = 0}.
    \end{align*}
    Then, since $\exp(-vt)$ is an isometry with its inverse $\exp(vt)$, we have that 
    \begin{align*}
        \nabla_{\dot\gamma_j(s)}\dot\gamma_j(s)|_{s = 0} = \exp(vt)_* (\nabla_{\exp(-vt)_*(\dot\gamma_j(s))}\exp(-vt)_*(\dot\gamma_j(s))|_{p_0}).
    \end{align*}
    On the other hand, by \eqref{sinh existence},
    \begin{align*}
        \exp(-vt)_*(\dot\gamma_j(s)) 
        =& \frac{d}{ds}\exp(-vt)(\gamma_j(s))\\
        =& \frac{d}{ds}\exp(-vt)\exp(k_j s)\exp(vt)(p_0)\\
        =& \frac{d}{ds} \exp([\exp(ad(-vt))(k_j)]s)(p_0)\\
        =& \frac{d}{ds} \exp([\mathrm{sh}(-a_j(v)t)p_j + \mathrm{ch}(-a_j(v)t)k_j]s)(p_0)\\
        =& \sigma(\mathrm{sh}(-a_j(v)t)p_j + \mathrm{ch}(-a_j(v)t)k_j)|_{exp(-vt)(\gamma_j(s))},
    \end{align*}
    and at $p_0 \in M$, 
    \begin{align*}
        \exp(-vt)_*(\dot\gamma_j(s))|_{s = 0} = \sigma(\mathrm{sh}(-a_j(v)t)p_j + \mathrm{ch}(-a_j(v)t)k_j)|_{p_0} = \mathrm{sh}(-a_j(v)t)\cdot\sigma(p_j)|_{p_0}.
    \end{align*}
    Therefore, by \eqref{relation between lie algebra and kill field}, \eqref{sectional curvature} and \eqref{basic frame relation 1}, we have  
    \begin{align*}
        &\nabla_{\exp(-vt)_*(\dot\gamma_j(s))}\exp(-vt)_*(\dot\gamma_j(s))|_{p_0}\\
        =& \mathrm{sh}(-a_j(v)t)\cdot\nabla_{\sigma(p_j)}\sigma(\mathrm{sh}(-a_j(v)t)p_j + \mathrm{ch}(-a_j(v)t)k_j)|_{p_0}\\
        =& \mathrm{sh}(-a_j(v)t)\cdot\Big[\sigma(p_j), \mathrm{sh}(-a_j(v)t)\sigma(p_j) + \mathrm{ch}(-a_j(v)t)\sigma(k_j)\Big]\Big|_{p_0}\\
        =& -\mathrm{sh}(-a_j(v)t)\cdot \mathrm{ch}(-a_j(v)t)\cdot\sigma([p_j, k_j])|_{p_0}\\
        =& \mathrm{sh}(-a_j(v)t)\cdot \mathrm{ch}(-a_j(v)t)\cdot \sum_{k = 0}^{r}a_j(v_k)\sigma(v_k)|_{p_0}. 
    \end{align*}
    Thereby, 
    \begin{align*}
        &\exp(vt)_*(\nabla_{\exp(-vt)_*(\dot\gamma_j(s))}\exp(-vt)_*(\dot\gamma_j(s))|_{p_0})\\
        =& \mathrm{sh}(-a_j(v)t)\cdot \mathrm{ch}(-a_j(v)t)\cdot \sum_{k = 0}^{r}a_j(v_k)\sigma(v_k)|_{\exp(vt)(p_0)}.
    \end{align*}
    The last step is followed by 
    \begin{align*}
        \exp(vt)_*(\sigma(v_k)|_{p_0}) 
        =& \frac{d}{ds} \bigg|_{s = 0} \exp(vt)\cdot \exp(v_ks)(p_0)\\
        =& \frac{d}{ds} \bigg|_{s = 0} \exp(vt)\cdot \exp(v_ks)\cdot \exp(-vt)\cdot \exp(vt)(p_0)\\
        =& \frac{d}{ds} \bigg|_{s = 0} \exp([\exp(ad(vt))v_k]s)(\exp(vt)(p_0))\\
        =& \frac{d}{ds} \bigg|_{s = 0} \exp(v_ks)(\exp(vt)(p_0))\\
        =& \sigma(v_k)|_{\exp(vt)(p_0)}.
    \end{align*}
    By the similar steps, we can get that for $1 \leq i \leq r$, 
    \begin{align*}
        \nabla_{\sigma(v_i)}\sigma(v_i)|_{\exp(vt)(p_0)} = 0.
    \end{align*}
\end{proof}

\subsection{Curvature and Laplacian Operators}
From \Cref{spherical frame}, we see that 
\begin{align*}
    &\bigg\{ \sigma(v_1), \cdots, \sigma(v_r), \frac{1}{\mathrm{sh}(a_1(v) t)}\sigma(k_1), \cdots, \frac{1}{\mathrm{sh}(a_{n - r} (v)t)}\sigma(k_{n - r})\bigg\}\bigg|_{\exp(vt)(p_0)}
\end{align*}
forms an orthonormal basis of $T_{\exp(v t)(p_0)}M$. Under this basis, we are able to write down the Ricci curvature and $\Delta_{g_0} d$ along the geodesic radius, $\exp(v t)(p_0)$, with respect to the root systems.
\begin{lemma}\label{laplacian distance function}
    Let $(M^n, g_0)$ be a simply connected irreducible symmetric space of non-compact type of rank $r$. Then, 
    \begin{equation}\label{laplace distance function ricci}
        \mathrm{Ric} = -\sum_{i = 1}^{n - r} (a_{i}(v))^2 g_0,
    \end{equation}
    and 
    \begin{equation}\label{laplace distance function itself}
        \Delta_{g_0} d|_{\exp(v t)(p_0)} = -\sum_{i = 1}^{n - r}\sum_{l = 1}^{r}\frac{\mathrm{ch}(-a_i(v)t)}{\mathrm{sh}(-a_i(v)t)}a_i(v_{l})\cdot (v_l, v),
    \end{equation}
    for any $v\in \mathfrak{a}$ with $(v, v) = 1$. 
\end{lemma}
\begin{proof}
For \eqref{laplace distance function ricci}, since simply connected irreducible symmetric spaces of non-compact type are all Einstein and $(v, v) = g_0(\sigma(v), \sigma(v))|_{p_0}$, it suffices to show that 
\begin{align*}
    \mathrm{Ric}(\sigma(v), \sigma(v)) = - \sum_{i = 1}^{n - r} (a_i(v))^2.
\end{align*}
By \eqref{sectional curvature}, we have that 
\begin{align*}
    \mathrm{Ric}(\sigma(v),\sigma(v))|_{p_0} =& -\sum_{i = 1}^{r}g_0(\sigma([[v,v_i], v]),\sigma(v_i))|_{p_0} - \sum_{i=1}^{n - r}g_0(\sigma([[v,p_i],v]), \sigma(p_i))|_{p_0}.
\end{align*}
By \eqref{metric = Killing form} and \eqref{inner product on lie algebra}, we have that 
\begin{align*}
    \mathrm{Ric}(\sigma(v),\sigma(v))|_{p_0} 
    =& -\sum_{i = 1}^{r}\lambda\langle[[v,v_i], v],v_i\rangle - \sum_{i=1}^{n - r}\lambda \langle [[v,p_i],v], p_i \rangle \\
    =& \sum_{i = 1}^{r}\lambda\langle[v_i,v],[v_i,v]\rangle+\sum_{i = 1}^{n - r}\lambda \langle [p_i,v],[p_i,v] \rangle \\
    =& \sum_{i=1}^{n-r}(a_i(v))^2\lambda \langle k_i, k_i \rangle = -\sum_{i=1}^{n-r}(a_i(v))^2\lambda ( k_i, k_i) \\ 
    =& -\sum_{i = 1}^{n - r} (a_i(v))^2.
\end{align*}
For \eqref{laplace distance function itself}, by \eqref{metric = Killing form} \eqref{inner product on lie algebra} and \Cref{spherical frame},
\begin{align*}
    \Delta_{g_0} d|_{\exp(vt)(p_0)} =& \sum_{i = 1}^{r} \sigma(v_i)\sigma(v_i)(d) + \sum_{i = 1}^{n - r}\frac{1}{\mathrm{sh}^2(a_i(v)t)}\sigma(k_{i})\sigma(k_{i})(d) - \sum_{i = 1}^{r}\nabla_{\sigma(v_i)}\sigma(v_i)(d)\\
    &- \sum_{i =1}^{n - r}\frac{1}{\mathrm{sh}^2(a_i(v)t)}\nabla_{\sigma(k_i)}\sigma(k_i)(d)|_{\exp(vt)(p_0)}\\
    =& \sum_{i = 1}^{n - r}-\frac{1}{\mathrm{sh}^2(a_i(v)t)}\nabla_{\sigma(k_i)}\sigma(k_i)(d)|_{\exp(vt)(p_0)}\\
    =& -\sum_{i = 1}^{n - r}\sum_{l = 1}^{r}\frac{\mathrm{ch}(-a_i(v)t)}{\mathrm{sh}(-a_i(v)t)}a_i(v_{l})\sigma(v_{l})(d)|_{\exp(vt)(p_0)}\\
    =& -\sum_{i = 1}^{n - r}\sum_{l = 1}^{r}\frac{\mathrm{ch}(-a_i(v)t)}{\mathrm{sh}(-a_i(v)t)}a_i(v_{l})\cdot g_0(\sigma(v_l), \sigma(v))|_{\exp(vt)(p_0)}\\
    =& -\sum_{i = 1}^{n - r}\sum_{l = 1}^{r}\frac{\mathrm{ch}(-a_i(v)t)}{\mathrm{sh}(-a_i(v)t)}a_i(v_{l})\cdot (v_l, v).
\end{align*} 
\end{proof}

From the above lemma, we see that for the rank 1 case, taking $v = v_1$, by the invariant of $\Delta_{g_0} d$ under the action of $K$ on $M$, we have 
\begin{equation}
    \lim_{d\rightarrow+\infty} \Delta_{g_0} d = \lim_{t \rightarrow + \infty} \Delta_{g_0} d|_{\exp(v_1 t)(p_0)} = \sum_{i = 1}^{n - 1} a_i(v_1),
\end{equation}
for the higher rank cases, we have that 
\begin{equation}
    \overline{\lim_{d\rightarrow +\infty}} \Delta_{g_0} d =  \overline{\lim_{t\rightarrow +\infty}} \Delta_{g_0} d|_{\exp(v t)(p_0)} = \sup_{v\in\mathfrak{a}, |v| = 1} \sum_{i = 1}^{n - r} \sum_{l = 1}^{r} a_i(v) \cdot \lambda \cdot (v_l, v).
\end{equation}

We can now state a criterion for the existence of the relative volume defined in \eqref{definition of relative volume} on simply connected symmetric spaces of non-compact type. Indeed, we will establish this existence result for Cartan-Hadamard manifolds that include such symmetric spaces as a special case.
\begin{lemma}\label{existence of relative volume}
    Let $(M^n, g_0)$ be a Cartan-Hadamard manifold. Let $g$ be another metric on $M$. Then, the relative volume defined in \eqref{definition of relative volume} exists and does not rely on the choice of compact exhaustion, if 
    \begin{equation}
        \|e^{\tau d(\cdot)}(g(\cdot) - g_0(\cdot))\|_{C^0} < +\infty,
    \end{equation}
    for some $\displaystyle \tau > \overline{\lim}_{d\rightarrow+\infty} \Delta_{g_0} d$.
\end{lemma}
\begin{proof}
    For $(M, g_0)$ being a Cartan-Hadamard manifold, we can establish a global normal coordinate at $p_0\in M$ with respect to $g_0$, such that $M\setminus\{p_0\} \cong \mathbb{R}_+ \times \mathbb{S}^{n - 1}$ and $g_0 = dr^2 + g_r$, where $r > 0$ is the distance function to $p_0\in M$ and $g_r$ is a family of metrics on $\mathbb{S}^{n - 1}$. Denote 
    \begin{align*}
        \hat{\tau}_0 := \overline{\lim_{d\rightarrow+\infty}}\Delta_g r = \overline{\lim_{d\rightarrow+\infty}} \Delta_g d. 
    \end{align*}
    It is straightforward to verify that
    \begin{align*}
        \partial_r \ln \left( \sqrt{\mathrm{det}(g_r)} \right) = \Delta_{g_0} r.
    \end{align*}
    Thus, for any fixed $\tau_1\in (\hat{\tau}_0, \tau)$, we always have
    \begin{align}\label{sphere area blow up}
        \sqrt{\mathrm{det} (g_r)} \leq C_1 e^{\tau_1 \cdot r},
    \end{align}
    where $C_1 > 0$ is a constant.
    
    On the other hand, let $\{\Omega_i\}_{i=1}^{+\infty}$ be an arbitrary compact exhaustion of $M$. For notational simplicity, under the normal coordinate, $(r, \theta)\in \mathbb{R}_{+} \times \mathbb{S}^{n - 1}$, define 
    \begin{align*}
        f(r, \theta) := \sqrt{\mathrm{det}(g)} - \sqrt{\mathrm{det}(g_0)}. 
    \end{align*}
    Then, 
    \begin{align*}
        \mathrm{Vol}(\Omega_i, g) - \mathrm{Vol}(\Omega_i, g_0) =& \int_{\Omega_i}d\mu_g - \int_{\Omega_i}d\mu_{g_0}\\ 
        =& \int_{\Omega_i} \sqrt{\mathrm{det}(g)} - \sqrt{\mathrm{det}(g_0)}\;dr\;d\theta\\
        =& \int_{\Omega_i} f(r, \theta)\; dr\; d\theta,
    \end{align*}
    where $dr,d\theta$ denotes the product of the radial measure and the standard volume form on $\mathbb{S}^{n-1}$.
    
    In order to show the existence of $\lim_{i \rightarrow +\infty} \int_{\Omega_{i}}f(r, \theta)dr d\theta$ for any arbitrary compact exhaustion $\{\Omega_i\}_{i = 1}^{+\infty}$, it suffices to show that $f(r, \theta)$ is absolutely integrable on $\mathbb{R}_+ \times \mathbb{S}^{n - 1}$. 
    Indeed, 
    \begin{align*}
         \int_{B(p_0, s)} |f(r, \theta)|\; dr\; d\theta =&  \int_{B(p_0, s)}|\sqrt{\mathrm{det}(g)} - \sqrt{\mathrm{det}(g_0)}|\; dr\; d\theta\\
        =&  \int_{0}^{+\infty} \int_{\mathbb{S}^{n - 1}}|\sqrt{\frac{\mathrm{det}(g)}{\mathrm{det}(g_0)}} - 1|\cdot |\sqrt{\mathrm{det}(g_0)}|\; d\theta\; dr\\
        \leq& C_2 \int_{0}^{+\infty} \int_{\mathbb{S}^{n - 1}}|g - g_0|\cdot \sqrt{\mathrm{det}(g_r)}\; d\theta\; dr\\
        \leq& C_3 \int_{0}^{+\infty} \int_{\mathbb{S}^{n - 1}} e^{-\tau r} e^{\tau_1 r} dr < +\infty,
    \end{align*}  
    where $C_2, C_3 > 0$ are two constants. Consequently, $f(r, \theta)$ is absolutely integrable on $\mathbb{R}_+ \times \mathbb{S}^{n - 1}$. Thus, 
    \begin{align*}
        V_{g_0}(g) = \lim_{i\rightarrow+\infty}(\mathrm{Vol}(\Omega_i, g) - \mathrm{Vol}(\Omega_i, g)) = \int_{\mathbb{R}_{+}\times \mathbb{S}^{n - 1}} f(r, \theta) \; dr\; d\theta.
    \end{align*}
\end{proof}
\begin{remark}
    In particular, for the hyperbolic space $\mathbb{H}^{n}$ with the standard hyperbolic metric $g_{\mathbb{H}^n}$, we know that $G = SO(1, n)$ and $K = SO(n)$. Taking $\lambda = \frac{1}{2(n - 1)}$ in \eqref{metric = Killing form}, it is straightforward to verify that it is of rank 1, $a_1(v) = \cdots = a_{n - 1}(v) = 1$ for $v\in \mathfrak{a}$ and $|\sigma(v)| = 1$, $\mathrm{Ric}[g_0] = (n - 1)g_0$, and  $\lim_{d\rightarrow+\infty}\Delta_{g_0} d = (n - 1)$. Therefore, the relative volume $\mathrm{Vol}_{g_{\mathbb{H}^{n }}}(g)$ exists, provided that $\|e^{\tau d} (g - g_{\mathbb{H}^n})\| < +\infty$ for some $\tau > n - 1$.
\end{remark}

Next, we recall the approach of I.4.A in \cite{Bi2006} for computing $\gamma$ defined in \eqref{two important quantities}.We first express the curvature operator, $\mathrm{Rm}$, in terms of Casimir operators associated with representations of the Lie algebra $\mathfrak{l}$. Then, by applying the representation theory of semisimple Lie algebras, we determine the largest eigenvalue of $\mathrm{Rm}$ which coincides with $\gamma$ in \eqref{two important quantities}. 

\begin{lemma}\label{curvature representation lemma}
Let $(M, g_0)$ be an irreducible symmetric space of non-compact type. Denote
\begin{align*}
    \mathrm{Rm}: T^{*}M &\otimes T^*M \rightarrow T^*M\otimes T^*M,\\
    h_{ij} &\mapsto R_{ikjl} g_0^{kk_1} g_0^{ll_1} h_{k_1l_1}.
\end{align*}
Then, for any $u, v \in \mathfrak{p}$, ,
\begin{align}\label{curvature representation 1}
    \omega^{-1}\circ\mathrm{Rm} \circ \omega (u\otimes v) = -\sum_{i}^{m} [e_i, u] \otimes [e_i, v],  
\end{align}
where $e_1, \cdots, e_m$ is an orthonormal basis of $\mathfrak{l}$ with respect to  $( \cdot , \cdot )$ defined in \eqref{inner product on lie algebra}
\begin{align}
    \omega: \mathfrak{p}\otimes \mathfrak{p} \rightarrow T^*_{p_0}M \otimes T^*_{p_0}M, \quad \omega (u\otimes v) = (\sigma(u))^{\flat}\otimes (\sigma(v))^{\flat}\Big|_{p_0}.
\end{align}
Moreover, 
\begin{align}\label{curvature representation}
    \omega^{-1}\circ\mathrm{Rm}\circ \omega = \frac{1}{2}(\mathfrak{C}(\mathfrak{l}, \mathfrak{p}\otimes\mathfrak{p}) - \mathfrak{C}(\mathfrak{l}, \mathfrak{p})),
\end{align}
where 
\begin{align*}
    \mathfrak{C}(\mathfrak{l}, \mathfrak{p}\otimes \mathfrak{p}) : \mathfrak{p} \otimes \mathfrak{p} &\rightarrow \mathfrak{p} \otimes \mathfrak{p},\\
u\otimes v &\mapsto -\sum_{i} [e_i, [e_i, u]] \otimes v - \sum_{i}u \otimes [e_i, [e_i, v]] - 2\sum_{i}[e_i, u]\otimes [e_i, v], \\
    \mathfrak{C}(\mathfrak{l}, \mathfrak{p}): \mathfrak{p} \otimes \mathfrak{p} &\rightarrow \mathfrak{p} \otimes \mathfrak{p}, \\
    u\otimes v &\mapsto -\sum_{i} [e_i, [e_i, u]] \otimes v - \sum_{i}u \otimes [e_i, [e_i, v]] 
\end{align*}
are the corresponding Casimir operators.
\end{lemma}
\begin{proof}
    Let $f_1, \cdots, f_n$ be an orthonormal basis of $\mathfrak{p}$.  By \eqref{metric = Killing form}, \eqref{sectional curvature} and $\eqref{inner product on lie algebra}$, we obtain 
    \begin{align*}
        \mathrm{Rm}(\omega(f_i \otimes f_j)) =&
        \mathrm{Rm} \left( \sigma(f_i)^{\flat} \otimes \sigma(f_j)^{\flat} \right)\Big|_{p_0}\\
        =& \sum_{k = 1, l = 1}^{n} R_{kilj} \sigma(f_k)^{\flat} \otimes \sigma(f_l)^{\flat}\Big|_{p_0}\\
        =& \sum_{k = 1, l = 1}^{n} \lambda\langle [f_k, f_i], [f_l, f_j] \rangle \sigma(f_k)^{\flat} \otimes \sigma(f_l)^{\flat}\Big|_{p_0}\\
        =& -\sum_{k = 1, l = 1}^{n}\sum_{t = 1}^{m} \lambda\langle [f_k, f_i], e_t \rangle \cdot \lambda\langle [f_l, f_j], e_t \rangle \sigma(f_k)^{\flat} \otimes \sigma(f_l)^{\flat}\Big|_{p_0}\\
        =& -\sum_{k = 1, l = 1}^{n}\sum_{t = 1}^{m} \lambda\langle [e_t, f_i], f_k \rangle \cdot \lambda\langle [e_t, f_j], f_l \rangle \sigma(f_k)^{\flat} \otimes \sigma(f_l)^{\flat}\Big|_{p_0}\\
        =& \omega \left( -\sum_{t = 1}^{m} \left(\sum_{k = 1}^{n} \lambda\langle [e_t, f_i], f_k \rangle f_k \right) \otimes \left( \sum_{l = 1}^{n} \lambda\langle [e_t, f_j], f_l \rangle f_l \right) \right)\\
        =& \omega \left( -\sum_{t = 1}^{m}[e_t, f_i] \otimes [e_t, f_j] \right),
    \end{align*}
    which implies \eqref{curvature representation 1}, and according to \eqref{curvature representation 1}, it is straightforward to verify \eqref{curvature representation}. 
\end{proof}
By the representation theory (\cite{Bi2006}-\cite{FH2013}), the Casimir operator is a scalar operator for an irreducible representation of a semisimple Lie algebra, and this scalar can be directly computed by the highest weight for this representation (see details in Chapter 22.2 of \cite{Hu2012}). On the other hand, the irreducibility of a symmetric space implies that the representation of the Lie algebra $\mathfrak{l}$ acting on $\mathfrak{p}$ is irreducible. Thus, $\mathfrak{C}(\mathfrak{l}, \mathfrak{p})$ is easy to handle. However, $\mathfrak{l}$ acting on $\mathfrak{p}\otimes \mathfrak{p}$ might not be irreducible, which is the difficult part. Luckily, we can use Steinberg's formula (see details in Chapter 24.4 of \cite{Hu2012}) to decompose the space $\mathfrak{p}\otimes \mathfrak{p}$ into irreducible subspaces with respect to $\mathfrak{l}$ action, and compute the Casimir operators, $\mathfrak{C}(\mathfrak{l}, \mathfrak{p} \otimes \mathfrak{p})$, for each irreducible subspace in the decomposition. Thus, we can figure out the largest eigenvalue of $\mathrm{Rm}$. For the rank 1 case, the largest eigenvalues of $\mathrm{Rm}$ are computed by O.Biquard in I.4.A of \cite{Bi2006}. We have listed them in \Cref{symmetric spaces of non-compact type of rank 1}. For simplicity, we take appropriate $\lambda$ in \eqref{metric = Killing form} such that $a_1(v_1) = 1$. 

By \Cref{laplacian distance function}, \Cref{curvature representation lemma} and the classification of simply connected irreducible symmetric spaces of non-compact type of rank 1 with an appropriate $\lambda$ in \eqref{metric = Killing form} such that $a_1(v_1) = 1$, we have the following table, 
\begin{table}[htbp]
\begin{tabular}{|l|l|l|l|l|l|}
\hline
& dim & roots & $\gamma$ & $\tau_0$ & $R_0$ \\ \hline
        $\mathbb{H}^m$ & m & $a_i(v_1) = 1\;\; \text{for }\; 1\leq m - 1$ & 1 & $m - 1$ & $-m(m - 1)$ \\ \hline
        $\mathbb{C}\mathbb{H}^m$ & 2m & $\begin{cases}
            a_i(v_1) = 1 & \text{for }\; 1\leq i \leq 2m - 2\\
            a_i(v_1) = 2 & \text{for }\; i = 2m - 1\\
        \end{cases}$ & 4 & $2m$ & $-2m(2m + 2)$ \\ \hline
        $\mathbb{H}\mathbb{H}^m$ & 4m & $\begin{cases}
            a_i(v_1) = 1 & \text{for }\; 1\leq i\leq 4m - 4\\
            a_i(v_1) = 2 & \text{for }\; 4m - 3\leq i \leq 4m - 1\\
        \end{cases}$ & $4$ & $4m + 2$ & $-4m(4m + 8)$ \\ \hline
        $\mathbb{O}\mathbb{H}^2 $ & 16 & $\begin{cases}
            a_i(v_1) = 1 & \text{for }\; 1\leq i\leq 8\\
            a_i(v_1) = 2 & \text{for }\; 9\leq i \leq 15\\
        \end{cases}$ & 4 & 22 & $-16\cdot (8 + 4\cdot 7)$  \\ \hline
\end{tabular}
\caption{symmetric spaces of non-compact type of rank 1}
\label{symmetric spaces of non-compact type of rank 1}
\end{table}
where 
\begin{itemize}
    \item $\mathbb{H}^m = SO(1, m)/SO(m)$ are real hyperbolic spaces;  
    \item $\mathbb{C}\mathbb{H}^m = SU(1, m)/U(m)$ are complex hyperbolic spaces;
    \item $\mathbb{H}\mathbb{H}^m = Sp(1, m)/Sp(1)Sp(m)$  are quaternionic hyperbolic spaces;
    \item $\mathbb{O}\mathbb{H}^2 = F^{-20}_4/Spin_9$ is the octonionic hyperbolic space,
\end{itemize}
and $\gamma$ and $\tau_0$ are defined in \eqref{two important quantities}. The roots of each symmetric space are given by the classification of their underlying real Lie algebra $\mathfrak{g}$.

\section{Normalized Ricci Flow and Normalized Ricci-DeTurck Flow}\label{sec: ricci_flow}

In this section, we introduce the normalized Ricci flow and the normalized Ricci--DeTurck flow and review basic stability results.
\subsection{Definitions}
Let $(M^n, g_0)$ be an Einstein manifold with scalar curvature $R_0$, and let $g$ be another Riemannian metric on $M$. The normalized Ricci flow (abbreviated as NRF) is a one-parameter family of metrics $\tilde g(t)$ (and likewise $\hat g(t)$ below) on $M$ evolving according to
\begin{align}\label{NRF}
    \left\{
        \begin{aligned} 
            \frac{\partial}{\partial t} \tilde{g}(t) & =-2\left(\mathrm{Ric}[\tilde{g}(t)] - \frac{R_0}{n} \tilde{g}(t)\right) \\ 
            \tilde{g}(0) &=g
        \end{aligned}
    \right.,
\end{align}
where $\mathrm{Ric}[\tilde{g}(t)]$ is the Ricci curvature with respect to $\tilde{g}(t)$. Since the linearization of the right-hand side of \eqref{NRF} is only weakly elliptic, it is difficult to analyze directly. We therefore consider the normalized Ricci--DeTurck flow (abbreviated as NRDF) with respect to the background metric $g_0$:
\begin{equation}\label{NRDF}
    \left\{
        \begin{aligned}
            \frac{\partial}{\partial t} \hat{g}(t) &= -2\left(\mathrm{Ric}[\hat{g}(t)] - \frac{R_0}{n} \hat{g}(t) \right)+\hat{\nabla}_i W_j+\hat{\nabla}_j W_i  \\
            \hat{g}(0) &= g
        \end{aligned}
    \right.,
\end{equation}
where $\hat{\nabla}$ is the covariant derivative with respect to $\hat{g}$, 
\begin{align}
    W_j = \hat{g}_{j k} \hat{g}^{p q} \left( \hat{\Gamma}_{pq}^k - \Gamma_{pq}^k \right),
\end{align}
and $\Gamma, \hat{\Gamma}$ are the Christoffel symbols with respect to the metrics $g_0$ and $\hat{g}$ respectively. $\tilde{g}$ and $\hat{g}$ can be related by a family of smooth diffeomorphisms $\Phi_t: M\rightarrow M$, which satisfy the following equation
\begin{equation}\label{gauge term diffeomorphism}
    \left\{
        \begin{aligned}
            \frac{\partial}{\partial t} \Phi_t(x) =& -W(\Phi_t(x), t)\\
            \Phi_0(x) =& \; \mathrm{id}(x)
        \end{aligned}
    \right.,
\end{equation}
where the component of $W$ is given by $W^i := \hat{g}^{ij} W_j$. It is straightforward to verify that $\tilde{g}(t) = \Phi_{t}^{*}\hat{g}(t)$. In particular, the linearization of the right-hand side of \eqref{NRDF} at the background metric $g_0$ is strictly elliptic. We denote the corresponding linearized operator by $L$. A direct computation gives
\begin{align}
    L = \Delta_L - 2\frac{R_0}{n},
\end{align}
where $\Delta_L$ is the Lichnerowicz Laplacian with respect to the metric $g_0$ on the symmetric 2-tensor $u$, i.e., 
\begin{align}
    \Delta_L u_{ij} := \Delta_{g_0} u_{ij} + 2R_{ipjq} u^{pq} - R_{iq} u^{q}_j - R_{jq}u^{q}_{i}.
\end{align}
Moreover, Einstein condition for $g_0$ implies 
\begin{align}
    L u_{ij} = \Delta_{g_0} u_{ij} + 2 R_{ipjq} u^{pq}.
\end{align}

The positivity of this operator plays a crucial role in the study of Einstein manifolds. To further analyze the geometric properties of Einstein manifolds, we introduce the notion of stability for Einstein metrics. This concept is closely tied to the spectral properties of the operator $L$.
\begin{definition}\label{strictly stable}
    A complete Riemannian manifold $(M^n, g_0)$ is called strictly stable if 
    \begin{align}
        \lambda_0 :=  \inf_h \frac{\displaystyle -\int_M  \langle L h, h \rangle d\mu_{g_0}}{\displaystyle \int_M |h|^2 d\mu_{g_0}} > 0,
    \end{align}
    where the infimum is taken among all nonzero symmetric 2-tensor $h$ such that 
    \begin{align*}
        \int_{M} \left( |h|^2 + |\nabla h|^2 \right) d\mu_{g_0} < + \infty.
    \end{align*}
\end{definition}
\begin{remark}
    R.Bamler showed that simply connected irreducible symmetric spaces of non-compact type are all strictly stable except $\mathbb{H}^2$ in Section 5.3 of  \cite{Ba2015}. 
\end{remark}

\subsection{Short-Time Existence}
We first recall the classical short-time existence results for the NRDF \eqref{NRDF}. Hamilton and Shi proved short-time existence for the NRF \eqref{NRF} on closed and complete manifolds, respectively \cite{Ha1982, Sh1989}. Moreover, Shi and Simon established the corresponding derivative estimates for complete manifolds \cite{Sh1989, Si2002}. Specifically, they proved that if $g_0$ is a complete metric with $\|\mathrm{Rm}\|_{C^0} \le k_0$, then there exists $\varepsilon>0$ such that, whenever $\|g-g_0\|_{C^0} \le \varepsilon$, there exists $T = T(n,k_0) > 0$ for which \eqref{NRDF} admits a solution $\hat g(t)$ on $t\in[0,T]$, with initial data $g$, satisfying
\begin{align}
    \|\hat g(t) - g_0\|_{C^0} \le 2\varepsilon,
\end{align}
and, for each integer $i\ge 1$,
\begin{align}
    \|\nabla^{i}\hat g(t)\|_{C^0}\le \frac{C(k_0,k_1,\dots,k_i)}{t^{i/2}}.
\end{align}
Here $k_0,k_1,\dots,k_i$ are positive constants controlling the curvature and its derivatives, namely
\begin{align}
    \|\nabla^{i}\mathrm{Rm}\|_{C^0} \le k_i.
\end{align}

\subsection{Long-Time Existence}
Regarding long-time existence, Ye established such a result for closed manifolds under curvature pinching conditions \cite{Ye1993}. Subsequently, Li--Yin and Schn\"urer--Schulze--Simon proved long-time existence for perturbations of hyperbolic space \cite{LY2010,SSS2010}, and Bamler treated perturbations of symmetric spaces of non-compact type \cite{Ba2015}. Qing, Shi, and Wu established long-time existence for perturbations of conformally compact Einstein manifolds \cite{QSW2013}. As observed in \cite{HJS2016, QSW2013}, both the NRF \eqref{NRF} and the NRDF \eqref{NRDF} preserve the conformal infinity.

In this subsection, we prove a long-time existence result for perturbations of strictly stable Einstein manifolds (\Cref{long time existence without weights}), a class that includes simply connected irreducible symmetric spaces of non-compact type. Following \cite{LY2010, QSW2013, Ye1993}, the key input is exponential decay of the metric difference in time. For convenience, we denote $\hat{h}(t,x) := \hat{g}(t,x) - g_0(x)$, or simply $\hat{h}$.

\begin{lemma}\label{local L^2 to C^0 estimate}
    Let $(M^n, g_0)$ be a strictly stable Einstein manifold satisfying 
    \begin{itemize}
        \item $\|\mathrm{Rm}\|_{C^{0}} \leq k_0$, for some $k_0 > 0$;
        \item $\displaystyle \inf_{x\in M} \mathrm{Vol}(B(x, 1), g_0) \geq v_0$, for some $v_0 > 0$;
        \item $\displaystyle \int_{M} e^{-\alpha d(x)} d\mu_{g_0}(x) < +\infty$, for some $\alpha > 0$.
    \end{itemize}
    Then, there exist sufficiently small $\varepsilon > 0$ and $\lambda_{\varepsilon} > 0$, such that for any $T > 0$ and any solution, $\hat{g}(t, \cdot)$, to the NRDF \eqref{NRDF}, starting at $g(\cdot)$, for $t\in [0, T]$, satisfying that 
    \begin{align*}
        \|\hat{g}(t, \cdot) - g_0(\cdot)\|_{C^{0}} \leq \varepsilon, \quad \text{for }\; t\in[0, T], \quad  \text{and} \quad \int_{M}|g - g_0|^{2} d\mu_{g_0} < +\infty, 
    \end{align*}
    we have that  
    \begin{align}
        \sup_{ x\in  M} |\nabla^{i} (\hat{g}(t, x) - g_0(x))|^2 \leq C e^{-2\lambda_{\varepsilon}t} \int_{M}|g - g_0|^2 d\mu_{g_0}, \quad \text{for any }\;  i \in \mathbb{N},
    \end{align}
    where $C$ is a constant depending on $t_0$, $k_0$, $v_0$, $n$, $i$. 
\end{lemma}

\begin{proof}
By direct computation, the evolution equation for $\hat{h}$ is given by (cf. Lemma 2.1 in \cite{Sh1989})
\begin{align}\label{evolution equation of metric difference}
    \frac{\partial}{\partial t} \hat{h} = \nabla_{p} (\hat{g}^{pq}\nabla_q \hat{h}) + 2\mathrm{Rm}(\hat{h}) + (\hat{g}^{-1} - g^{-1}) \ast \mathrm{Rm} \ast \hat{h} + \hat{g}^{-1} \ast \hat{g}^{-1}\ast \nabla \hat{h} \ast \nabla \hat{h},
\end{align}
where $\mathrm{Rm}(\hat{h})_{ij} := R_{ikjl}\hat{h}^{kl}$ and the notation, $A*B$, denotes linear combinations of terms formed by contraction of two tensors, $A$ and $B$, with respect to $g_0$. Following \cite{LY2010}, for points $x, y \in M$ and parameters $t, s, r > 0$, define the auxiliary function
\begin{align}
    \xi(s, t, y, x) = - \frac{ \left( d_r(y, x) \right)^2}{(2 - C_0 \varepsilon)(t - s)}, 
\end{align}
where 
\begin{align*}
    d_r(y, x) = 
    \begin{cases}
        0, & \text{if } \mathrm{dist}(y, x) \leq r \\
        \mathrm{dist}(y, x), & \text{otherwise} 
    \end{cases},
\end{align*}
and $C_0$ is a constant such that 
\begin{align}
    \frac{\partial}{\partial s} \xi + \frac{1}{2} \hat{g}^{pq}(\nabla_p \xi) \cdot (\nabla_q \xi) \leq 0.
\end{align}
Here, the covariant derivative is taken for the variable $y$. Let 
\begin{align}
    J(s, t, x) := \int_{M} |\hat{h}(s, y)|^2\cdot e^{\xi(s, t, y, x)} d\mu_{g_0}(y).
\end{align}
Since $\int_M e^{-\alpha d(x, p_0)} d\mu_{g_0} < +\infty$, the above $J(s, t, x) < +\infty$ for any fixed $ t > s \geq 0$ and $x\in M$. Taking the derivative for $J$ with respect to $s$, we obtain
\begin{align}
    \frac{\partial}{\partial s} J = \int_M 2 \hat{h} \frac{\partial \hat{h}}{\partial s} \cdot e^\xi + | \hat{h} |^2 \cdot \frac{\partial \xi}{\partial s} e^\xi d\mu_{g_0}. 
\end{align}
Combining with the evolution equation of $\hat{h}$ \eqref{evolution equation of metric difference}, we derive
\begin{equation}
    \begin{split}
        \frac{\partial}{\partial s} J &\leq 
        -2\int_M \hat{g}^{pq} \nabla_p (e^{\frac{\xi}{2}}\hat{h}) \cdot \nabla_p (e^{\frac{\xi}{2}}\hat{h}) d\mu_{g_0} 
        + 4 \int_{M}\mathrm{Rm}(\hat{h}, \hat{h}) e^{\xi} d\mu_{g_0}\\
        &+ \int_M \left(\frac{\partial}{\partial s}\xi + \frac{1}{2}\hat{g}^{pq}(\nabla_p \xi)(\nabla_q \xi) \right) \cdot |\hat{h}|^2 e^{\xi} d\mu_{g_0} + 2\|\hat{g}^{-1} - g^{-1} \|_{C^{0}} \cdot\|\mathrm{Rm}\|_{C^{0}} \cdot J\\
        &+ 2\cdot \|\hat{h}\|_{C^{0}}\cdot \|\hat{g}^{-1}\|_{C^{0}}\cdot \int_{M}|\nabla(e^{\frac{\xi}{2}}\hat{h})|^2d\mu_{g_0}\\
        &\leq -(2 - C(n)\varepsilon) \int_M \big(|\nabla (e^{\frac{\xi}{2}}\hat{h})|^2 - 2\mathrm{Rm}(\hat{h}, \hat{h})e^{\xi}\big)d\mu_{g_0}+C(n, k_0) \cdot \varepsilon\cdot J\\
        &\leq \big(-(2 - C(n)\varepsilon) \lambda_0 + C(n, k_0)\varepsilon \big)J,
    \end{split}
\end{equation}
where $\mathrm{Rm}(\hat{h}, \hat{h}) = R_{ijkl} \hat{h}^{ik} \hat{h}^{jl}$ and $\lambda_0$ is defined in  \Cref{strictly stable}. Let $\lambda_{\varepsilon} = ((2 - C(n)\varepsilon) \lambda_0 - C(n, k_0)\varepsilon)$ and take $\varepsilon > 0$ small enough such that $\lambda_\varepsilon > 0$, we obtain
\begin{equation}\label{formla 1}
    \begin{split}
        J(s, t, x) \leq& e^{-2\lambda_{\varepsilon} s} J(0, t, x)\\
        =& e^{-2\lambda_{\varepsilon} s} \int_{M} |\hat{h}(0, y)|^2\exp \left(-\frac{\left(d_r(y, x) \right)^2}{(2 - C_0\varepsilon)t} \right)d\mu_{g_0}(y)\\
        \leq& e^{-2\lambda_{\varepsilon} s} \int_M|\hat{h}(0, y)|^2 d\mu_{g_0}(y).
    \end{split}
\end{equation}
On the other hand, by direct computation, $|\hat{h}|^2$ satisfies
\begin{equation}
    \begin{split}
        \frac{\partial }{\partial t} |\hat{h}|^2 \leq&\; \hat{g}^{pq} \nabla_p \nabla_q |\hat{h}|^2 - (2 - 2|\hat{g}^{-1} - g^{-1}| - |\hat{g}^{-1}|^2 \cdot |\hat{h}|)\cdot |\nabla \hat{h}|^2 + (2 + |\hat{g}^{-1} - g^{-1}|)k_0\cdot |\hat{h}|^2\\
        \leq&\; \hat{g}^{pq}\nabla_p \nabla_q |\hat{h}|^2 - (2 - 2C(n) \varepsilon - C(n, \varepsilon)\varepsilon)\cdot |\nabla \hat{h}|^2 + (2 + C(n) \varepsilon)\cdot k_0\cdot |\hat{h}|^2.
    \end{split}
\end{equation}
Modifying $\varepsilon > 0$ such that $2 - 2C(n) \varepsilon - C(n, \varepsilon)\varepsilon > 0$, we have that 
\begin{align}
    \frac{\partial}{\partial t} |\hat{h}|^2 \leq \hat{g}^{pq} \nabla_p \nabla_q |\hat{h}|^2 + (2 + C(n)\varepsilon)\cdot k_0 \cdot |\hat{h}|^2.
\end{align}
Thus, by the De Giorgi-Nash-Moser iteration and \eqref{formla 1}, for any $(t, x)\in (0, T) \times M$, we have
\begin{equation}\label{De Giorgi-Nash-Moser iteration}
    \begin{split}
        \sup_{(t^{\prime}, x^{\prime})\in[t - \frac{r^2}{4}, t]\times B(x, \frac{r}{2})}|\nabla^{i}\hat{h}(t^{\prime}, x^{\prime})|\leq&\; C(i, k_0, v_0, r)\int_{t - r^2}^{t} \int_{B(x,r)}|\hat{h}(s, y)|^2 d\mu_{g_0}(y) ds\\  
        \leq&\; C(i, k_0, v_0, r) \int_{t - r^2}^{t}J(s, t, x) ds\\
        \leq&\; C(i, k_0, v_0, r)\int_{t - r^2}^{t} e^{-2\lambda_{\varepsilon} s}ds\cdot \int_{M}|\hat{h}(0, y)|^2d\mu_{g_0}(y)\\
        \leq&\; C(i, k_0, v_0, r)e^{-2\lambda_{\varepsilon} t}\int_{M} |\hat{h}(0, y)|^2 d\mu_{g_0}(y).
    \end{split}
\end{equation}
Then, the lemma follows. 
\end{proof}

By a standard contradiction argument, long-time existence follows.
\begin{lemma}\label{long time existence without weights}
    Let $(M^n, g_0)$ be a strictly stable Einstein manifold satisfying
    \begin{itemize}
        \item $\|\mathrm{Rm}\|_{C^0} \leq k_0$, for some $k_0 > 0$;
        \item $\displaystyle \inf_{x\in M} \mathrm{Vol}(B(x, 1), g_0) \geq v_0$, for some $v_0 > 0$;
        \item $\displaystyle \int_{M} e^{-\alpha d(x)} d\mu_{g_0}(x) < +\infty$, for some $\alpha > 0$.
    \end{itemize}
    Then, there exist sufficiently small $\varepsilon_0 >0$, $\varepsilon_1 > 0$ and $\lambda_{\varepsilon_0} > 0$ such that for any Riemannian metric $g$ on $M$, satisfying that 
    \begin{align*}
        \|g - g_0\|_{C^{0}} \leq \varepsilon_0 \quad \text{and} \quad \int_{M}|g - g_0|^2 d\mu_{g_0} \leq \varepsilon_1, 
    \end{align*}
    the solution, $\hat{g}(t, \cdot)$, to the NRDF \eqref{NRDF}, starting at $g(\cdot)$, exists for all the time, satisfying that 
    \begin{align}\label{long time existence with weight result 1}
        \|\hat{g}(t, \cdot) - g_0(\cdot)\|_{C^{i}} \leq C \varepsilon_1 e^{-\lambda_{\varepsilon_0} t}, \quad \text{for any }\; t\in [t_0, +\infty), 
    \end{align}
    where $t_0 > 0$ is an arbitrary given number and C is a constant depending on $i, k_0, v_0, t_0$. Moreover, $\hat{g}(t)$ exponentially converges to $g_0$ as $t\rightarrow +\infty$. 
\end{lemma}
\begin{proof}
    By \Cref{local L^2 to C^0 estimate}, following the contradiction arguments in \cite{QSW2013, Ye1993}, the lemma follows.
\end{proof}

\section{Long-Time Behavior of the NRDF}\label{sec: longtime}

\subsection{Spatial Decay Estimates}
As shown in \cite{HJS2016, QSW2013}, for a conformally compact background metric, the short-time evolution of the NRDF \eqref{NRDF} preserves the spatial decay of the initial metric relative to the background metric. The following lemma follows directly from their arguments and shows that their results also hold for certain other manifolds. Again, define $\hat{h}(t,x) := \hat{g}(t,x) - g_0(x)$ and, when no ambiguity arises, write $\hat{h}$ for brevity.

\begin{lemma}\label{short time existence with weight}
    Let $(M^n, g_0)$ be a complete manifold with a pole $p_0 \in M$, satisfying $\|\nabla^{i}\mathrm{Rm}\|_{C^0}\leq k_i$ for some constants $k_i > 0$ with $i = 0, 1, 2$. Then, for any $\delta > 0$, there exist positive constants $\varepsilon_0(n, \delta)$, $T(n, k_0)$, and $C(\varepsilon_0, n, k_0, k_1, k_2, \delta)$ such that, for any Riemannian metric $g$ on $M$, if
    \begin{align*}
        \| e^{\delta d}(g - g_0)\|_{C^1} \leq \varepsilon_0, 
    \end{align*}
    there exists a solution, $\hat{g}(t,\cdot)$, to the NRDF \eqref{NRDF}, starting at $g(\cdot)$, for $t\in [0, T]$, satisfying that 
    \begin{equation}\label{result of short time existence with weight}
        \begin{split}
            &\|e^{\delta d(\cdot)}(\hat{g}(t, \cdot) - g_0(\cdot))\|_{C^1} \leq C( \varepsilon_0, n, k_0, k_1, k_2, \delta),\quad \text{for}\; t\in [0, T]\;\; \text{and}\\
            &\|e^{\delta d(\cdot)}\nabla^2 (\hat{g}(t, \cdot) - g_0(\cdot))\|_{C^0} \leq \frac{C(\varepsilon_0, n, k_0, k_1, k_2, \delta)}{\sqrt{t}},\quad \text{for } \; t \in (0, T], 
        \end{split}
    \end{equation}
    where $d(x)$ is the distance function between $x\in M$ and the pole $p_0$ with respect to the background metric $g_0$.
\end{lemma}
\begin{proof}
 According to the results of \cite{Sh1989, Si2002}, there exists $T = T(n, k_0) > 0$, such that the solution to the NRDF~\eqref{NRDF}, $\hat{g}(t,\cdot)$, with initial data $g(\cdot)$,  exists on $t \in [0, T]$ and satisfies 
\begin{equation}\label{extra2}
    \begin{split}
        &\|\hat{h}\|_{C^{1}} \leq \varepsilon_1(\varepsilon_0, n, k_0, k_1),\quad \text{for }\; t\in [0, T]\;\;\; \text{ and} \\
        &\|\nabla^{2}\hat{h}\|_{C^{0}} \leq \frac{\varepsilon_2(\varepsilon_0, n, k_0, k_1, k_2)}{\sqrt{t}},\quad \text{for }\; t\in (0, T],
    \end{split}
\end{equation}
where $\varepsilon_1, \varepsilon_2 > 0$ and $\lim_{\varepsilon_0\rightarrow 0} \varepsilon_1 = 0$, $\lim_{\varepsilon_0 \rightarrow 0} \varepsilon_2 = 0$. Following the approach in \cite{HJS2016,QSW2013}, we denote
\begin{align}
    \varphi := |e^{\delta d} \hat{h}|^2 + |e^{\delta d}\nabla \hat{h}|^2 + t|e^{\delta d}\nabla^2 \hat{h}|^2.
\end{align}
Since the distance function $d(x)$ is not smooth at $p_0$, instead of working on $(0, T]\times M$, we restrict our analysis to $(0, T]\times \Omega_1$, where $\Omega_{1}: = \{x\in M\;|\; d(x) \geq 1\}$. Then the evolution inequality satisfied by $\varphi$ on $(0, T]\times \Omega_1$ is
\begin{equation}
    \begin{split}
        \frac{\partial}{\partial t} \varphi 
        \leq&\; \hat{g}^{pq} \nabla_p \nabla_q \varphi - 2\delta\cdot \hat{g}^{pq}\cdot (\nabla_p d)\cdot (\nabla_q \varphi)\\
        &-\bigg[2 - 2C(n)\varepsilon_1 - C(\sup_{x\in \Omega_1}|\nabla^2 d(x)|, \delta, \|\hat{g}\|_{C^0})\frac{\varepsilon_1^2}{2} - C(\varepsilon_1, k_0, k_1, k_2)\frac{\varepsilon_1^2}{2}\\
        &-C(\delta, \varepsilon_1, k_0, k_1, k_2)(\varepsilon_1 + \varepsilon_2)\bigg]\cdot (|e^{\delta d}\nabla \hat{h}|^2 + |e^{\delta d}\nabla^2 \hat{h}|^2 + t|e^{\delta d}\nabla^3\hat{h}|^2)\\
        &+\bigg[ C(\sup_{x\in\Omega_{1}}|\nabla^2 d(x)|, \varepsilon_1, \delta)\frac{1}{2\varepsilon_1^2} + C(\varepsilon_1, k_0, k_1, k_2)\frac{1}{2\varepsilon_1^2} + C(\delta, \varepsilon_1, k_0, k_1, k_2) \bigg] \varphi.
    \end{split}
\end{equation} 
By the Hessian comparison theorem, 
\begin{align*}
    \sup_{x\in \Omega_1}|\nabla^2 d(x)| \leq C(k_0),\quad \text{for }\; x \in \Omega_1. 
\end{align*}
Choosing $\,\varepsilon_0 > 0$ sufficiently small so that
\[
2 - 2C(n)\varepsilon_1 - C(|\nabla^2 d|, \delta, |\hat{g}|)\frac{\varepsilon_1^2}{2} - C(\varepsilon_1, k_0, k_1, k_2)\frac{\varepsilon_1^2}{2} - C(\delta, \varepsilon_1, k_0, k_1, k_2)(\varepsilon_1 + \varepsilon_2) > 0,
\]
we obtain
\[
\frac{\partial}{\partial t} \varphi \leq \hat{g}^{pq} \nabla_p \nabla_q \varphi - 2\delta\, \hat{g}^{pq} (\nabla_p d)(\nabla_q \varphi) + C(\delta, \varepsilon_1, k_0, k_1, k_2)\varphi,
\]
for all $(t,x) \in (0,T] \times \Omega_1$. Define
\[
\tilde{\varphi} := e^{-C(\delta, \varepsilon_1, k_0, k_1, k_2)t} \, \varphi.
\]
Then,
\[
\frac{\partial}{\partial t} \tilde{\varphi} \leq \hat{g}^{pq} \nabla_p \nabla_q \tilde{\varphi} - 2\delta \, \hat{g}^{pq} (\nabla_p d)(\nabla_q \tilde{\varphi}).
\]
By the maximum principle (Theorem 4.3 in \cite{QSW2013}, originally due to \cite{LT1992}), we obtain
\begin{equation}\label{extra3}
    \begin{split}
        \sup_{(t, x)\in(0, T]\times\Omega_{1}} \tilde{\varphi}(t, x)
        &\leq \max\Big\{\sup_{x\in \Omega_{1}} \tilde{\varphi}(0, x),\; \sup_{(t, x)\in (0, T]\times\partial \Omega_{1}} \tilde{\varphi}(t, x)\Big\}. 
    \end{split}
\end{equation}
In addition, 
\begin{align}\label{extra4}
    \sup_{x\in\Omega_{1}} \tilde{\varphi}(0, x)
    \leq \|\tilde{\varphi}(0, \cdot)\|_{C^0}
    \leq \|e^{\delta d}(g - g_0)\|_{C^1}
    \leq \varepsilon_0,
\end{align}
and
\begin{equation}\label{extra5}
    \begin{split}
        \sup_{(t, x) \in (0, T] \times \partial \Omega_{1}} \tilde{\varphi}(t, x)
        &\leq e^{2\delta} \sup_{(t, x)\in (0, T]\times \partial\Omega_1} \big(|\hat{h}|^2 + |\nabla \hat{h}|^2 + t|\nabla^2 \hat{h}|^2\big) \\
        &\leq C(\varepsilon_0, n, k_0, k_1, k_2, \delta),
    \end{split}
\end{equation}
where the last inequality of the above follows from \eqref{extra2}. Therefore, according to \eqref{extra3}, \eqref{extra4} and \eqref{extra5}, we obtain
\begin{align}\label{extra7}
    \sup_{(t, x)\in (0, T]\times \Omega_1} \varphi(t, x) \leq \sup_{(t, x)\in (0, T]\times \Omega_1}e^{C(\delta, \varepsilon_1, k_0, k_1, k_2) t}\tilde{\varphi}(t, x)\leq C(\varepsilon_0, n, k_0, k_1, k_2, \delta, T(n, k_0)).
\end{align}
On the other hand, by (\ref{extra2}), 
\begin{equation}\label{extra6}
    \begin{split}
        \sup_{(t, x)\in (0, T]\times (M\setminus\Omega_1)} \varphi(t, x)
        &\leq e^{2\delta} \sup_{(t, x)\in (0, T]\times (M\setminus \Omega_1)} \big(|\hat{h}|^2 + |\nabla \hat{h}|^2 + t|\nabla^2 \hat{h}|^2\big) \\
        &\leq C(\varepsilon_0, n, k_0, k_1, k_2, \delta).
    \end{split}
\end{equation}
Combining \eqref{extra7} and \eqref{extra6} yields the lemma.
\end{proof}

\begin{remark}\label{cartan hadamard einstein condition}
    In what follows, the background manifold is assumed to be a Cartan-Hadamard Einstein manifold satisfying $\|\mathrm{Rm}\|_{C^0} \leq k_0$ for some $k_0 > 0$. It therefore suffices to establish
    \begin{align}\label{long time existence with weight formula 6}
        \|\nabla \mathrm{Rm}\|_{C^0} \leq k_1
        \quad \text{and} \quad
        \|\nabla^2 \mathrm{Rm}\|_{C^0} \leq k_2,
    \end{align}
    for suitable constants $k_1, k_2 > 0$. In fact, under the Cartan-Hadamard assumption, $(M, g_0)$ has positive injectivity radius. Hence, by Theorem 11.4.3(Anderson (1990)) of \cite{Pp2016}, there exists $r_0 > 0$ such that for every $p \in M$, one can construct harmonic coordinate charts $\varphi_p: B(p, r_0) \to U \subset M$ in which $g_0 \in C^{1,\alpha}(B(p, r_0))$ with respect to the Euclidean metric. Since $g_0$ is Einstein, elliptic regularity implies that $g_0 \in C^m(B(p, r_0))$ for all $m > 1$, which in turn yields \eqref{long time existence with weight formula 6}. Therefore, the above lemma applies even without explicitly assuming boundedness of $\|\nabla \mathrm{Rm}\|_{C^0}$ and $\|\nabla^2 \mathrm{Rm}\|_{C^0}$ for Cartan-Hadamard Einstein manifolds with $\|\mathrm{Rm}\|_{C^0}\leq k_0$.
\end{remark}

By imposing suitable additional conditions, the spatial decay of the initial metric is preserved under the NRDF \eqref{NRDF} in the long-time sense as well.
\begin{lemma}\label{long time existence with short time weight}
    Let $(M^n, g_0)$ be a strictly stable Cartan-Hadamard Einstein manifold with $\|\mathrm{Rm}\|_{C^0}\leq k_0$, for some $k_0 > 0$, and suppose $\tau_0$ defined in \eqref{two important quantities} exists. Then, for any $\delta > \tau_0$, there exists $\varepsilon_0 > 0$, such that for any Riemannian metric $g$ on $M$, if 
    \begin{align}\label{long time existence with short time weight condition 1}
        \|e^{\delta d}(g - g_0)\|_{C^1} \leq \varepsilon_0, 
    \end{align}
    the solution, $\hat{g}(t, \cdot)$, to the NRDF \eqref{NRDF}, starting at $g(\cdot)$, exists for all the time, satisfying that 
    \begin{align}\label{long time existence with short time weight result}
        |\hat{g}(t, x) - g_0(x)| \leq C(n, \varepsilon_0, \tau_0, k_0, \delta, T) e^{-\delta d(x)}, \quad \text{ on }\; [0, T]\times M,
    \end{align}
    for any $T > 0$. 
\end{lemma}
\begin{proof}
    Taking a normal coordinate at some fixed point, $p_0 \in M$, as in \Cref{existence of relative volume} with $\tau_1 = \frac{\delta + \tau_0}{2} \in (\tau_0, \delta)$, we obtain
    \begin{align}\label{long time existence with short time weight formula 1}
        \int_{M} e^{-\delta d(x)} d\mu_{g_0}(x) < C(n, k_0)\int_{0}^{+\infty} e^{-(\delta - \tau_1)r}dr < C(n, \tau_0, k_0, \delta).
    \end{align}
    Furthermore, by \eqref{long time existence with short time weight condition 1} and \eqref{long time existence with short time weight formula 1}, we obtain 
    \begin{align}\label{long time existence with short time weight formula 2}
        \int_M |\hat{h}(0, x)|^2 d\mu_{g_0}(x) \leq \varepsilon_0 \cdot \int_M e^{-2\delta d(x)} d\mu_{g_0}(x) = C(n, \tau_0, k_0, \delta)\cdot \varepsilon_0.
    \end{align}
    In addition, under the Cartan-Hadamard assumption, the volume of unit balls is uniformly lower bounded. According to \Cref{long time existence without weights} and \eqref{extra2}, taking $\varepsilon_0 > 0$ small enough, the solution to the NRDF \eqref{NRDF}, $\hat{g}(t, \cdot)$, with initial data $g(\cdot)$, exists for $t\in [0, +\infty)$ and there exist $\lambda_{\varepsilon_0} > 0$ and $C(n, \tau_0, k_0, \delta) > 0$ such that 
    \begin{align}\label{long time existence with short time weight formula 3}
        |\nabla^{i}(\hat{g}(t, x) - g_0(x))| < C(n, \tau_0, k_0, \delta)\varepsilon_0 e^{-\lambda_{\varepsilon_0} t}, \quad \text{for }\; (t, x)\in [0, +\infty)\times M,\; i = 0, 1. 
    \end{align}
    By direct computation, the evolution equation of $|e^{\delta d}\hat{h}|^2$ is
    \begin{equation}\label{long time existence with short time weight formula 4}
        \begin{split}
        \frac{\partial}{\partial t} |e^{\delta d}\hat{h}|^2 
        \leq&\; \hat{g}^{pq} \nabla_p \nabla_q \left(  |e^{\delta d} \hat{h}|^2 \right) - 2 \tau \cdot \hat{g}^{pq} \cdot (\nabla_p d) \cdot \nabla_q  \left( |e^{\delta d}\hat{h}|^2 \right)\\
        &- A\cdot|\nabla (e^{\delta d}\hat{h})|^2 + 2B\cdot |e^{\delta d}\hat{h}|^2,
        \end{split}
    \end{equation}
    where
    \begin{align*}
        A =&\; 2 -2|\hat{g}^{-1} - g^{-1}| - C(n, \delta, k_0)\cdot|\hat{g}^{-1}|^2\cdot |\hat{h}|,\\ 
        B =&\; \delta^2 - \tau_0 \delta + 2\gamma + C(n, \delta, k_0)\cdot |\hat{g}^{-1} - g^{-1}| + C(n)\cdot |\hat{g}^{-1}|^2 \cdot |\hat{h}| + \delta (\tau_0 - \Delta_g d).
    \end{align*}
    By \eqref{long time existence with short time weight formula 3}, there exists $C(n, \varepsilon_0, \tau_0, k_0, \delta) > 0$ such that
    \begin{align}\label{long time existence with short time weight formula 5}
        B < C(n,\varepsilon_0,\tau_0, k_0, \delta)
    \end{align}
    on $[0, +\infty) \times \Omega_1$, where $\Omega_1 := \{x\in M \;|\; d(x) \geq 1\}$. Moreover, by taking $\varepsilon_0 > 0$ further smaller, we can guarantee that $A > 0$. Now, define
    \begin{align*}
        \varphi(t, x) := e^{-2C(n,\varepsilon_0,\tau_0, k_0, \delta)t}|e^{\delta d(x)}\hat{h}(t, x)|^2. 
    \end{align*}
    Combining with \eqref{long time existence with short time weight formula 4}, \eqref{long time existence with short time weight formula 5} and $A > 0$, we derive that 
    \begin{align}
        \frac{\partial}{\partial t} \varphi(t, x)
        \leq&\; \hat{g}^{pq} \nabla_p \nabla_q \left(\varphi(t, x) \right) - 2 \tau \cdot \hat{g}^{pq} \cdot (\nabla_p d) \cdot \nabla_q  \left(\varphi(t, x) \right) 
    \end{align}
    on $[0, +\infty) \times \Omega_1$. By the maximum principle (Theorem 4.3 in \cite{QSW2013}), we obtain that 
    \begin{equation}\label{long time existence with short time weight formula 6}
        \begin{split}
            \sup_{(t, x)\in [0, +\infty)\times\Omega_1}\varphi(t, x) \leq \max\Big\{\sup_{x\in \Omega_1}\varphi(0, x), \sup_{(t, x)\in [0, +\infty)\times\partial \Omega_1} \varphi(t, x) \Big\}. 
        \end{split}
    \end{equation}
    By \eqref{long time existence with short time weight condition 1}, 
    \begin{align}\label{long time existence with short time weight formula 7}
        \sup_{x\in \Omega_1} \varphi(0, x) \leq \|e^{\delta d}(g - g_0)\|_{C^0} \leq \varepsilon_0. 
    \end{align}
    By \eqref{long time existence with short time weight formula 3}, 
    \begin{equation}\label{long time existence with short time weight formula 8}
        \begin{split}
            \sup_{(t,x)\in [0, +\infty) \times \partial \Omega_1} \varphi(t, x) \leq&\; \sup_{(t, x)\in [0, +\infty) \times \partial \Omega_1} e^{-2C(n,\varepsilon_0, \delta,\tau_0, k_0)t}|e^{\delta d(x)}\hat{h}(t, x)|^2\\
            \leq&\; \sup_{(t, x)\in [0, +\infty)\times M}e^{\delta} |\hat{h}(t, x)|^2\leq C(n, \varepsilon_0, \tau_0, k_0, \delta).
        \end{split}
    \end{equation}
    Combining \eqref{long time existence with short time weight formula 6}, \eqref{long time existence with short time weight formula 7} and \eqref{long time existence with short time weight formula 8}, we obtain
    \begin{equation}\label{long time existence with short time weight formula 9}
        \sup_{(t, x)\in [0, +\infty) \times \Omega_1} \varphi(t, x) \leq C(n, \varepsilon_0, k_0, \tau_0, \delta). 
    \end{equation}
    In addition, by \eqref{long time existence with short time weight formula 3}, 
    \begin{equation}\label{long time existence with short time weight formula 10}
        \begin{split}
            \sup_{(t, x)\in [0, +\infty) \times (M\setminus \Omega_1)} \varphi(t, x) \leq&\; \sup_{(t, x)\in [0, +\infty) \times (M\setminus \Omega_1)} e^{-2C(n, \varepsilon_0, \delta, \tau_0, k_0)t}|e^{\delta d(x)}\hat{h}(t, x)|^2\\
            \leq&\; e^{\delta}\cdot \sup_{(t, x)\in [0, +\infty) \times (M\setminus \Omega_1)} |\hat{h}(t, x)|^2\\
            \leq&\; e^{\delta}\cdot \sup_{(t, x)\in [0. +\infty) \times M} |\hat{h}(t, x)|^2 \leq C(n, \varepsilon_0, k_0, \tau_0, \delta).
        \end{split}
    \end{equation}
    By \eqref{long time existence with short time weight formula 9} and \eqref{long time existence with short time weight formula 10}, we obtain
    \begin{align*}
        \sup_{(t, x)\in [0, +\infty) \times M}\varphi(t, x) \leq C(n,\varepsilon_0, k_0, \tau_0, \delta), 
    \end{align*} 
    which implies
    \begin{align}\label{long time existence with short time weight formula 11}
        |\hat{g}(t, x) - g_0(x)| \leq C(n,\varepsilon_0, k_0, \tau_0, \delta) e^{2 C(n,\varepsilon_0, k_0, \tau_0, \delta)t} e^{-\delta d(x)}.
    \end{align}
    \eqref{long time existence with short time weight result} follows from \eqref{long time existence with short time weight formula 11}. 
\end{proof}
Furthermore, the spatial decay and the time decay can be separated under stronger conditions. 
\begin{lemma}\label{long time existence with weight}
    Let $(M^n, g_0)$ be a strictly stable Cartan-Hadamard Einstein manifold with $\|\mathrm{Rm}\|_{C^0} \leq k_0$, for some $k_0 > 0$, and suppose that $\tau_0$ defined in \eqref{two important quantities} exists.
    Then, for any 
    \begin{align}
        \tau \in \left(\frac{\tau_0}{2} - \sqrt{\frac{\tau_0^2}{4} - 2\gamma}, \frac{\tau_0}{2} + \sqrt{\frac{\tau_0^2}{4} - 2\gamma} \right) \quad \text{and} \quad \delta > \tau_0,
    \end{align}
    there exist $\varepsilon_0 > 0$ and $\lambda > 0$, such that for any Riemannian metric $g$ on $M$, if 
    \begin{align}\label{long time eistence with weight condition1}
        \|e^{\delta d}(g - g_0)\|_{C^{1}} \leq  \varepsilon_0,
    \end{align}
    the solution, $\hat{g}(t, \cdot)$, to the NRDF \eqref{NRDF}, starting at $g(\cdot)$, exists for all the time, satisfying that 
    \begin{align}\label{long time with weight result 1}
        |\nabla^i(\hat{g}(t,x) - g_0(x))| \leq C(n,\varepsilon_0,\tau_0, k_0, \delta) e^{-\lambda t}\cdot e^{-\tau d(x)},
    \end{align}
    for $t\in [0, +\infty)$, $i = 0, 1$, and 
    \begin{align}\label{long time with weight result 2}
        |\nabla^2(\hat{g}(t,x) - g_0(x))| \leq C(n,\varepsilon_0,\tau_0, k_0, \delta)(1 + \frac{1}{\sqrt{t}}) e^{-\lambda t}\cdot e^{-\tau d(x)},
    \end{align}
    for $t\in (0, +\infty)$.
\end{lemma}
\begin{proof}
    By direct computation, we obtain 
\begin{equation}\label{long time existence without weights formula 2}
    \begin{split}
        \frac{\partial}{\partial t} e^{2\lambda t}|e^{\tau d}\hat{h}|^2 
        \leq&\; \hat{g}^{pq} \nabla_p \nabla_q \left( e^{2\lambda t} |e^{\tau d} \hat{h}|^2 \right) - 2 \tau \cdot \hat{g}^{pq} \cdot (\nabla_p d) \cdot \nabla_q  \left( e^{2\lambda t} |e^{\tau d}\hat{h}|^2 \right)\\
        &- A\cdot e^{2\lambda t}|\nabla (e^{\tau d}\hat{h})|^2 + 2(B + \lambda)\cdot e^{2\lambda t} |e^{\tau d}\hat{h}|^2,
    \end{split}
\end{equation}
where 
\begin{align*}
    A =&\; 2 -2|\hat{g}^{-1} - g^{-1}| - C(n, \tau, k_0)\cdot|\hat{g}^{-1}|^2\cdot |\hat{h}|,\\ 
    B =&\; \tau^2 - \tau_0 \cdot \tau + 2\gamma + C(n, \tau, k_0)\cdot |\hat{g}^{-1} - g^{-1}| + C(n)\cdot |\hat{g}^{-1}|^2 \cdot |\hat{h}| + \tau (\tau_0 - \Delta_g d).
\end{align*}
As in \Cref{long time existence with short time weight}, we have  
\begin{align}
    \int_M |\hat{h}(0, x)|^2 d\mu_{g_0}(x) \leq \varepsilon_0 \cdot \int_M e^{-2\delta d(x)} d\mu_{g_0}(x) = C(n, \tau_0, k_0, \delta)\cdot \varepsilon_0.
\end{align}
By \Cref{long time existence without weights} and Cartan-Hadamard assumption, taking $\varepsilon_0 > 0$ small enough, the solution to the NRDF, $\hat{g}(t, \cdot)$, with initial data $g(\cdot)$, exists for $t\in [0, +\infty)$ and there exist $\lambda_{\varepsilon_0} > 0$ and $C(n, \delta, \tau_0, k_0, t_0) > 0$ such that 
\begin{align}\label{long time eistence with weight formula 1}
    |\nabla^{i}(\hat{g}(t, x) - g_0(x))|^2 < C(n, \tau_0, k_0, \delta, t_0)\varepsilon_0 e^{-2\lambda_{\varepsilon_0} t} \quad \text{for} \quad t\in [t_0, +\infty),\; i = 0, 1, 2. 
\end{align}
 On the other hand, for any
\begin{align}
    \tau \in \left(\frac{\tau_0}{2} - \sqrt{\frac{\tau_0^2}{4} - 2\gamma} ,\, \frac{\tau_0}{2} + \sqrt{\frac{\tau_0^2}{4} - 2\gamma} \right),
\end{align}
by choosing $\lambda > 0$ and $\varepsilon_0 > 0$ sufficiently small and $d_0 > 0$ sufficiently large, we are able to guarantee 
\begin{align*}
    A > 0, \quad \lambda + B < 0, \quad \text{and} \quad \lambda < \lambda_{\varepsilon_0},
\end{align*}
for $(t, x)\in [t_0, +\infty)\times\Omega_{d_0}$, where 
\begin{align*}
    \Omega_{d_0} := \{x\in M\; |\; d(x) > d_0\}.
\end{align*}
Thus, according to \eqref{long time existence without weights formula 2}, for $(t, x)\in [t_0, +\infty)\times\Omega_{d_0}$, we have
\begin{align}
    \frac{\partial}{\partial t} \left( e^{2\lambda t}|e^{\tau d}\hat{h}|^2 \right) 
    \leq \hat{g}^{pq} \nabla_p \nabla_q \left( e^{2\lambda t}|e^{\tau d} \hat{h}|^2 \right) - 2\tau \cdot \hat{g}^{pq} \cdot (\nabla_p d) \cdot \nabla_q \left( e^{2\lambda t} |e^{\tau d}\hat{h}|^2 \right).
\end{align}
By the maximum principle (Theorem 4.3 in \cite{QSW2013}) and \eqref{long time eistence with weight formula 1}, we obtain 
\begin{equation}\label{long time weight estimate 1}
    \begin{split}
        &\sup_{(t, x)\in[t_0, +\infty) \times \Omega_{d_0}} e^{2\lambda t}|e^{\tau d(x)}\hat{h}(t, x)|^2 \\
        \leq& \max\Big\{\sup_{x\in\Omega_{d_0}} |e^{\tau d(x)}\hat{h}(t_0, x)|^2, \sup_{(t, x)\in[t_0, +\infty)\times \partial\Omega_{d_0}}e^{2\lambda t}e^{2\tau d_0}|\hat{h}(t, x)|^2\Big\}\\
        \leq& \max\Big\{\sup_{x\in\Omega_{d_0}} |e^{\tau d(x)}\hat{h}(t_0, x)|^2, C(n, \delta, \tau_0, k_0, v_0, t_0)\sup_{(t, x)\in[t_0, +\infty)\times \partial\Omega_{d_0}}e^{2\lambda t}e^{2\tau d_0} e^{-2 \lambda_{\varepsilon_0} t}\Big\}.
    \end{split}
\end{equation}
Fix $t_0 < T(n,k_0)$, where $T(n,k_0)$ is as in \Cref{short time existence with weight}. Then, \Cref{short time existence with weight} yields
\begin{equation}\label{long time weight estimate 2}
    |\hat{h}(t_0,x)| \leq C(n,\tau,k_0) e^{-\tau d(x)}.
\end{equation}
Combining \eqref{long time eistence with weight formula 1}, \eqref{long time weight estimate 1} and \eqref{long time weight estimate 2} yields the zeroth-order covariant derivative part of \eqref{long time with weight result 1}. That is 
\begin{align}\label{long time existence with weight formula 4}
    |\hat{h}(t, x)| \leq C(n,\varepsilon_0,\tau_0, k_0, \delta) e^{-\lambda t}\cdot e^{-\tau d(x)}\;\; \text{for } t\in [0, +\infty). 
\end{align}
For the estimate of the first and second order covariant derivatives of $\hat{h}$, we argue as \eqref{De Giorgi-Nash-Moser iteration}. Under the Cartan-Hadamard condition, the volume of unit balls is uniformly lower bounded. Then, the De Giorgi-Nash-Moser iteration yields 
    \begin{align}\label{De Giorgi-Nash-Moser iteration 2} 
        |\nabla^{i}\hat{h}(t, x)|\leq\sup_{(t^{\prime}, x^{\prime})\in [t - \frac{r^2}{4}, t]\times B(x, \frac{r}{2})}|\nabla^i \hat{h}(t^{\prime}, x^{\prime})|\leq C(i, k_0, r)\cdot \sup_{(t^{\prime}, x^{\prime})\in [t - r^2, t]\times B(x, r)}|\hat{h}(t^{\prime}, x^{\prime})|. 
    \end{align}
Fixing $r < 2\sqrt{t_0}$, together with \eqref{long time existence with weight formula 4}, we obtain that 
\begin{align}\label{long time existence with weight formula 5}
    |\nabla^i(\hat{g}(t,x) - g_0(x))| \leq C(n,\varepsilon_0,\tau_0, k_0, \delta) e^{-\lambda t}\cdot e^{-\tau d(x)},\quad \text{for}\; t\in [t_0, +\infty)\; \text{ and }\; i = 0, 1, 2.
\end{align}
Combining \Cref{short time existence with weight} and \eqref{long time existence with weight formula 5} completes the proof.
\end{proof}

Next, we shall investigate the long-time behavior of $W$ in \eqref{NRDF}. 

\begin{lemma}\label{The evolution equation of the deTurck term}
    Let $(M^n, g_0)$ be a strictly stable Cartan-Hadamard Einstein manifold satisfying the same conditions in \Cref{long time existence with weight}. Then, for any 
    \begin{align}\label{extra1}
        \delta \in \left( 0, \frac{\tau_0}{2} + \sqrt{\frac{\tau_0^2}{4} - \frac{R_0}{n}} \right) \cap \left( \tau_0, \tau_0 + 2\sqrt{\frac{\tau_0^2}{4} - 2\gamma} \right),
    \end{align}
    there exist $\varepsilon_0 > 0$ and $\hat{\lambda} > 0$, such that for any Riemannian metric $g$ on $M$, if 
    \begin{align*}
         \|e^{\delta d}(g - g_0)\|_{C^1} \leq \varepsilon_0,
    \end{align*}
    the solution, $\hat{g}(t, \cdot)$, to the NRDF \eqref{NRDF}, starting at $g(\cdot)$, exists for all the time, satisfying that 
    \begin{align}\label{The evolution equation of the deTurck term result 1}
        |W(t, x)|\leq C(n,\varepsilon_0,\tau_0, k_0, \delta) e^{-\hat{\lambda} t} e^{-\delta d(x)}, 
    \end{align}
    for $t\in [0, +\infty)$, and
    \begin{align}\label{The evolution equation of the deTurck term result 2}
        |\nabla^iW(t, x)|\leq C(n,\varepsilon_0,\tau_0, k_0, \delta, i, t_0) e^{-\hat{\lambda} t} e^{-\delta d(x)},
    \end{align}
    for $t\in [t_0, +\infty)$, where $W^{i} = \hat{g}^{ab}(\hat{\Gamma}_{ab}^{i} - \Gamma_{ab}^{i})$ and $t_0 > 0$. 
\end{lemma}

\begin{proof}
The evolution equation of the $W^{i}$ under the NRDF \eqref{NRDF} is the following 
\begin{align}
    \frac{\partial}{\partial t} W^{i} = \Delta_{\hat{g}} W^{i} + \frac{R_0}{n}W^{i} + (\hat{Ric} - \frac{R_0}{n} \hat{g})_{pq} \hat{g}^{ip} \hat{g}^{qr}W_{r} + Q(\hat{h}),
\end{align}
where 
\begin{equation}
    \begin{split}
        Q(\hat{h}) 
        =&\; (\hat{\mathrm{Ric}} - \frac{R_0}{n}\hat{g}) * \hat{g}^{-1}* \hat{g}^{-1}* \hat{g}^{-1}*\nabla\hat{g} + \hat{g}^{-1} * \nabla\hat{g} * \nabla W\\
        &+ \hat{g}^{-1}* \hat{g}^{-1} * \nabla\hat{g} * \nabla\hat{g} * W,
    \end{split}
\end{equation}
$\hat{\mathrm{Ric}}$ is the Ricci curvature for the metric, $\hat{g}$, and $A*B$ denotes linear combinations of terms formed by contraction of two tensors, $A$ and $B$, with respect to $g_0$. Moreover, we have that 
\begin{equation*}
    \begin{split}
        \frac{\partial}{\partial t} e^{\hat{\lambda} t}e^{\delta d}|W| \leq&\; \Delta_{\hat{g}}  \left( e^{\hat{\lambda} t}e^{\delta d} |W| \right) - 2\delta \cdot \hat{g}^{pq}\cdot (\nabla_p d) \cdot \nabla_q \left( e^{\hat{\lambda} t}e^{\delta d}|W| \right) \\ 
        &+ \bigg[ \hat{\lambda} + \delta^2 - \tau_0 \delta + \frac{R_0}{n} + C(n, |\hat{g}^{-1}|, |\nabla\hat{g}|, |\nabla^2 \hat{g}|)(|\nabla\hat{g}| + |\nabla^2 \hat{g}|)\\
        &+ |\hat{\mathrm{Ric}} - \frac{R_0}{n}\hat{g}| + \delta\cdot (\tau_0 - \Delta_g d) + |\hat{g}^{-1} - g^{-1}|\cdot (\delta^2|\nabla d| + \delta|\nabla^2 d|) \bigg] \cdot e^{\hat{\lambda} t}e^{\delta d}|W|\\ 
        &+ C(n, |\hat{g}^{-1}|, |\nabla \hat{g}|)\cdot e^{\hat{\lambda}t}e^{\delta d} \cdot |\nabla \hat{g}|\cdot (|\nabla^2\hat{g}| + |\hat{\mathrm{Ric}} - \frac{R_0}{n}\hat{g}|).
    \end{split}
\end{equation*}
Together with the fact that 
\begin{align}
    |\hat{\mathrm{Ric}} - \frac{R_0}{n}\hat{g}| = |\hat{\mathrm{Ric}} - \mathrm{Ric}|
    \leq  C(n, |\nabla\hat{g}|, |\hat{g}^{-1} - g^{-1}|, |\hat{g}^{-1}|)\cdot (|\nabla \hat{g}| + |\nabla^2 \hat{g}|),  
\end{align}
we obtain
\begin{align}
    \frac{\partial }{\partial t} e^{\hat{\lambda} t} e^{\delta d}|W| \leq \Delta_{\hat{g}} \left( e^{\hat{\lambda} t}e^{\delta d}|W| \right) - 2\delta \cdot \hat{g}^{pq}\cdot (\nabla_p d) \cdot \nabla_q  \left( e^{\hat{\lambda} t}e^{\delta d}|W| \right) + A \cdot e^{\hat{\lambda} t} e^{\delta d}|W| + B,
\end{align}
where 
\begin{align*}
    A =&\; \hat{\lambda} +   \delta^2 - \tau_0 \delta + \frac{R_0}{n} + \delta\cdot (\tau_0 - \Delta_g d)\\
    &+ C(n, \delta, |\hat{g}^{-1}|, |\hat{g}^{-1} - g^{-1}|, |\nabla\hat{g}|, |\nabla^2 \hat{g}|, |\nabla^2 d|)\cdot(|\nabla\hat{g}|\\
    &+ |\nabla^2 \hat{g}| + |\hat{g}^{-1} - g^{-1}|)\cdot(|\nabla\hat{g}| + |\nabla^2 \hat{g}| + |\hat{g}^{-1} - g^{-1}|),\\
    B =&\; C(n, |\hat{g}^{-1}|, |\hat{g}^{-1} - g^{-1}|, |\nabla \hat{g}|)\cdot e^{\hat{\lambda} t}e^{\delta d} \cdot |\nabla \hat{g}|\cdot (|\nabla^2\hat{g}| + |\nabla \hat{g}|).
\end{align*}
On the other hand, we see if $\delta$ satisfies (\ref{extra1}), then,
\begin{align*}
    \frac{\delta}{2} \in \left( \frac{\tau_0}{2}, \frac{\tau_0}{2} + \sqrt{\frac{\tau_0^2}{4} - 2\gamma} \right) \subset \left( \frac{\tau_0}{2} - \sqrt{\frac{\tau_0^2}{4} - 2\gamma}, \frac{\tau_0}{2} + \sqrt{\frac{\tau_0^2}{4} - 2\gamma} \right).
\end{align*}
By applying \Cref{long time existence with weight} with $\tau = \frac{\delta}{2}$, we obtain $\varepsilon_0 > 0$ and $\lambda > 0$ such that if 
\begin{align*}
    \|e^{\delta d}(g - g_0)\|_{C^{1}} \leq \varepsilon_0,
\end{align*} 
then, there exists a solution \( \hat{g}(t, \cdot) \) to the NRDF \eqref{NRDF}, with initial data \( g(\cdot) \), for all \( t \in [0, +\infty) \), satisfying that  
\begin{align}
    |\nabla^{i}(\hat{g}(t, x) - g(x))| \leq C(n, \tau_0, k_0, \delta, t_0) \varepsilon_0 e^{-\lambda t} e^{-\frac{\delta}{2} d(x)},
\end{align}
for $t\in [t_0, +\infty)$ and $i = 0, 1, 2$, where $t_0$ is an arbitrary positive number. Thus, 
\begin{align}\label{The evolution equation of the deTurck term formula 1}
    B < C(n,\varepsilon_0,\tau_0, k_0, \delta, t_0)\cdot e^{(\hat{\lambda} - 2\lambda)t}, \quad \text{for }\; t\in [t_0, +\infty).
\end{align}
In addition, for
\begin{align*}
    \delta \in \left(0, \frac{\tau_0}{2} + \sqrt{\frac{\tau_0^2}{4} - \frac{R_0}{n}} \right) \cap \left(\tau_0, \tau_0 + 2\sqrt{\frac{\tau_0^2}{4} - 2\gamma} \right),
\end{align*}
we can always make $\varepsilon_0$ further smaller, $d_0 > 0$ sufficiently larger and $\hat{\lambda} > 0$ sufficiently small such that $A < 0$ for $x \in \Omega_{d_0} := \{x\in M\; | \; d(x) > d_0\}$ and $ \hat{\lambda} < \lambda$. Together with \eqref{The evolution equation of the deTurck term formula 1}, we obtain 
\begin{equation}\label{evolution equation of deturck term formula 20}
    \begin{split}
        \frac{\partial}{\partial t } e^{\hat{\lambda} t} e^{\delta d}|W| \leq&\; \Delta_{\hat{g}} \left( e^{\hat{\lambda} t}e^{\delta d}|W| \right) - 2\delta \cdot \hat{g}^{pq}\cdot (\nabla_p d) \cdot \nabla_q \left( e^{\hat{\lambda} t}e^{\tau d}|W| \right)\\
        &+  C(n,\varepsilon_0,\tau_0, k_0, \delta, t_0) e^{ - \hat{\lambda} t}, 
    \end{split}
\end{equation}
for $(t, x)\in [t_0, +\infty) \times \Omega_{d_0}$. Consider $\psi := e^{\hat{\lambda} t} e^{\delta d} |W| + \frac{ C(n,\varepsilon_0,\tau_0, k_0, \delta, t_0)}{\hat{\lambda}} e^{ - \hat{\lambda} t}$. Combining with \eqref{evolution equation of deturck term formula 20}, we derive
\begin{align}
    \frac{\partial}{\partial t} \psi \leq \hat{g}^{pq} \nabla_p \nabla_q \psi - 2\delta \cdot \hat{g}^{pq}\cdot (\nabla_p d)\cdot (\nabla_q \psi),\;\;\; \text{ on }\; [t_0, +\infty)\times \Omega_{d_0}.
\end{align}
By the maximum principle (Theorem 4.3 in \cite{QSW2013}), we obtain 
\begin{equation}\label{deTurck term estimate 1}
    \begin{split}
        \sup_{(t, x)\in [t_0, +\infty)\times \Omega_{d_0}} \psi(t, x) \leq&\; \max\Big\{\sup_{x\in \Omega_{d_0}} e^{\hat{\lambda} t_0 } e^{\delta d(x)}|W(t_0, x)|, \sup_{(t, x)\in [t_0, +\infty)\times \partial \Omega_{d_0}} e^{\hat{\lambda} t} e^{\delta d_0} |W(t,x)|\Big\}\\
        & + \frac{C(n, \varepsilon_0, \tau_0, k_0, \delta, t_0)}{\hat{\lambda}}.
    \end{split}
\end{equation}
On the other hand, by \Cref{long time existence with weight} with $\tau = \frac{\delta}{2}$, 
\begin{equation}\label{deTurck term formula 4}
    \begin{split}
         |W(t, x)| =&\; |\hat{g}^{-1} * \hat{g}^{-1} * \nabla\hat{g}| = C(n) |\hat{g}^{-1}|^2\cdot |\nabla \hat{g}|\\
        \leq&\; C(n,\varepsilon_0,\tau_0, k_0, \delta) e^{-\lambda t} e^{-\frac{\delta}{2}d(x)}, \;\;\; \text{ for }\; (t, x)\in  [0, +\infty)\times  M,
    \end{split}
\end{equation}
which implies that 
\begin{align}\label{The evolution equation of the deTurck term formula 2}
    \sup_{(t, x)\in [t_0, +\infty) \times \partial \Omega_{d_0}} e^{\hat{\lambda}t}e^{\delta d_0}|W(t, x)| \leq C(n,\varepsilon_0,\tau_0, k_0, \delta)e^{(\hat{\lambda} - \lambda)t} e^{\frac{\delta}{2}d_0}\leq C(n,\varepsilon_0,\tau_0, k_0, \delta). 
\end{align}
Fix $t_0 > 0$ with $t_0 < T(n, k_0)$, where $T(n, k_0)$ is the number as in \Cref{short time existence with weight}. Then, \Cref{short time existence with weight} implies that  
\begin{align}\label{deTurck Term estimate 2}
    \sup_{(t, x)\in [0, t_0]\times M} e^{\hat{\lambda} t} e^{\delta d(x)}|W(t, x)| = \sup_{(t, x)\in [0, t_0]\times M} e^{\hat{\lambda} t } e^{\delta d(x)} |\hat{g}^{-1} * \hat{g}^{-1} * \nabla\hat{g}|\leq C(\varepsilon_0, n, k_0, k_1, k_2, \delta).
\end{align}
By \eqref{deTurck term estimate 1}, \eqref{The evolution equation of the deTurck term formula 2} and \eqref{deTurck Term estimate 2}, we obtain that
\begin{align}\label{deTurck term formula 3}
    \sup_{(t, x)\in [0, +\infty)\times \Omega_{d_0}} e^{\hat{\lambda} t} e^{\delta d(x)} |W(t, x)| \leq C(n,\varepsilon_0,\tau_0, k_0, \delta). 
\end{align}
By \eqref{deTurck term formula 4}, we obtain that 
\begin{align}\label{deTurck term formula 5}
    \sup_{(t, x)\in [0, +\infty)\times (M\setminus\Omega_{d_0})} e^{\hat{\lambda}t}e^{\delta d(x)}|W(t, x)| \leq e^{\delta d_0} \sup_{(t, x)\in [0, +\infty) \times M} e^{\hat{\lambda}t}|W(t, x)|\leq C(n,\varepsilon_0,\tau_0, k_0, \delta).
\end{align}
The zeroth-order estimate \eqref{The evolution equation of the deTurck term result 1} follows from \eqref{deTurck term formula 3} and \eqref{deTurck term formula 5}. The higher-order estimate \eqref{The evolution equation of the deTurck term result 2} is obtained via the De Giorgi-Nash-Moser iteration, as in \eqref{De Giorgi-Nash-Moser iteration 2} in \Cref{long time existence with weight}.
\end{proof}

\subsection{Scalar Curvature Preservation}
Let \(\hat R(t,x)\) and \(\tilde R(t,x)\) (or simply \(\hat R\) and \(\tilde R\) when there is no ambiguity) denote the scalar curvatures of \(\hat g(t,x)\) and \(\tilde g(t,x)\), respectively. We shall show that if the scalar curvature of the initial metric is greater than \(R_0\), then, \(\hat R\) and \(\tilde R\) remain greater than \(R_0\). This follows from the maximum principle, as in the previous arguments. We then prove that $|\hat R - R_0|$ decays exponentially in both time and space. Throughout this section, we use $|\cdot|_{\tilde{g}}$ and $|\cdot|_{\hat{g}}$ to denote pointwise norms with respect to metrics $\tilde{g}$ and $\hat{g}$, respectively.
\begin{lemma}\label{The evolution equation of the scalar curvature}
    Let $(M^n, g_0)$ be a strictly stable Cartan-Hadamard Einstein manifold satisfying the same conditions in \Cref{long time existence with weight}, Then, for any 
     \begin{align*}
        \delta \in \left( 0, \frac{\tau_0}{2} + \sqrt{\frac{\tau_0^2}{4} - 2\frac{R_0}{n}} \right) \cap \left( \tau_0, \tau_0 + 2\sqrt{\frac{\tau_0^2}{4} - 2\gamma} \right),
    \end{align*}
    there exist $\varepsilon_0 > 0$ and $\lambda_R > 0$, such that for any Riemannian metric $g$ on $M$, if 
    \begin{align*}
         \|e^{\delta d}(g - g_0)\|_{C^1} \leq \varepsilon_0,
    \end{align*}
    the solution $\hat{g}(t, \cdot)$ to the NRDF \eqref{NRDF}, starting at $g(\cdot)$, exists for all the time, satisfying that if $R \geq R_0$, then, 
    \begin{align}\label{The inequality of the scalar curvature}
        \hat{R} \geq R_0,
    \end{align}
    where $R$ is the scalar curvature of $g$. Furthermore,
    \begin{align}\label{scalar curvature estimate with weight}
        |\hat{R}(t, x) - R_0(x)|\leq C(n,\varepsilon_0,\tau_0, k_0, \delta)\cdot (1 + \frac{1}{\sqrt{t}}) e^{-\lambda_R t} e^{-\delta d(x)}.
    \end{align} 
\end{lemma}
\begin{proof}

The long-time existence of the NRDF \eqref{NRDF} follows directly from \Cref{long time existence with weight} with $\tau = \frac{\delta}{2}$. In this lemma, we refine the constants $\varepsilon_0 > 0$ and $\lambda > 0$ from \Cref{long time existence with weight}. The proof is divided into three steps: the first establishes \eqref{The inequality of the scalar curvature}, while the remaining two address \eqref{scalar curvature estimate with weight}.

\textbf{Step 1:} In this step, we prove the non-negativity of $\hat{R} - R_0$, assuming $R \geq R_0$. We consider the equation for $\tilde{R} - R_0$, since it is tidier. By direct computation, we obtain
\begin{align}\label{evolution equation of scalar curvature without weight}
    \frac{\partial}{\partial t} (\tilde{R} - R_0) = \Delta_{\tilde{g}}(\tilde{R} - R_0) + 2|\tilde{Ric} - \frac{R_0}{n} \tilde{g}|_{\tilde{g}}^2 + 2\frac{R_0}{n} (\tilde{R} - R_0),
\end{align}
which implies that 
\begin{align}
    \frac{\partial}{\partial t} (\tilde{R} - R_0) \geq&\; \Delta_{\tilde{g}}(\tilde{R} - R_0) + 2\frac{R_0}{n} (\tilde{R} - R_0).
\end{align}
By the maximum principle (Theorem 4.3 in \cite{QSW2013}), we have $\tilde{R} \geq R_0$. Furthermore, by the invariance of the scalar curvature under the  diffeomorphisms defined by \eqref{gauge term diffeomorphism}, we have $\hat{R} \geq R_0$. 

\textbf{Step 2:}
In this step, we shall show that 
\begin{align}\label{short time curvature}
    |\hat{R}(t, x) - R_0| \leq C(n, \varepsilon_0, \delta, k_0)(1 + \frac{1}{\sqrt{t}})e^{-\delta d(x)},\quad \text{for }\; t\in (0, T], 
\end{align}
where the above $T$ is as in \Cref{short time existence with weight}. In fact, by the fact
\begin{equation}\label{evolution equation of scalar curvature formula 1}
    \begin{split}
        &|\hat{R} - R_0|\\
        =&\; |\nabla^2\hat{g}* (\hat{g} + (\hat{g}^{-1} - g^{-1}))\\
        &+ (\nabla\hat{g} * \nabla\hat{g} * \hat{g}^{-1} * \hat{g}^{-1} + (\hat{g}^{-1} - g^{-1}) * \nabla^2\hat{g}) * (\hat{g} + (\hat{g}^{-1} - g^{-1}))\\
        &+ (\hat{g}^{-1} - g^{-1}) * \mathrm{Ric}|\\
        \leq&\; C(n, \|\hat{g}^{-1} - g^{-1}\|_{C^0}, \|\hat{g}^{-1}\|_{C^0}, \|\nabla\hat{g}\|_{C^0}, \|\mathrm{Rm}\|_{C^0})\cdot (|\nabla\hat{g}| + |\nabla^2 \hat{g}| + |\hat{g}^{-1} - g^{-1}|)
    \end{split}
\end{equation}
and \Cref{short time existence with weight}, we obtain \eqref{short time curvature}.

Note that we cannot directly apply \Cref{long time existence with weight} onto \eqref{evolution equation of scalar curvature formula 1} to establish the long-time part of \eqref{scalar curvature estimate with weight}, due to the mismatch between the spatial decay rate $\tau$ in \Cref{long time existence with weight} and $\delta$ in this lemma.

\textbf{Step 3:}
In this step, we address the long-time part of \eqref{scalar curvature estimate with weight}. Again, We proceed via the NRF \eqref{NRF}. Let $\tilde{g}(t,x)$ denote a solution to the NRF \eqref{NRF} with initial data $g$. The metrics $\tilde{g}$ and $\hat{g}$ are related by diffeomorphisms $\Phi_t : M \to M$ defined in \eqref{gauge term diffeomorphism}, via
\begin{align}
    \tilde{g}(t, \cdot) = \Phi_t^{*} \hat{g}(t, \cdot).
\end{align}
Denote the Ricci curvatures of $\hat{g}(t,x)$ and $\tilde{g}(t,x)$ by $\hat{\mathrm{Ric}}(t,x)$ and $\tilde{\mathrm{Ric}}(t,x)$, respectively. Then,
\begin{equation}\label{relations of the scalar curvature of NRF and NRDF}
    \tilde{R}(t,x) = \hat{R}(t,\Phi_t(x))\;\; \text{ and }\;\;\;
    \tilde{\mathrm{Ric}}(t,x) = \Phi_t^{*}\bigl(\hat{\mathrm{Ric}}\bigr)(t,x).
\end{equation}
By direct computation, the evolution inequality for $e^{\lambda_R t} e^{\delta d} (\tilde{R} - R_0)$ is 
\begin{equation}\label{the evolution equation of the scalar curvature formula}
    \begin{split}
        \frac{\partial}{\partial t} e^{\lambda_R t}e^{\delta d} (\tilde{R} - R_0) \leq&\; \Delta_{\tilde{g}} \left[ e^{\lambda_R t}e^{\delta d} (\tilde{R} - R_0) \right] - 2\delta \tilde{\nabla}d \cdot \tilde{\nabla} \left[ e^{\lambda_R t}e^{\delta d} (\tilde{R} - R_0) \right]\\ 
        &+ A\cdot e^{\lambda_R t}e^{\delta d}(\tilde{R} - R_0) + 2e^{\lambda_R t} e^{\delta d} \left| \tilde{\mathrm{Ric}} - \frac{R_0}{n}\tilde{g} \right|_{\tilde{g}}^2,
    \end{split}
\end{equation}
where 
\begin{align*}
    A = \big(\delta^2 - \tau_0\delta + 2\frac{R_0}{n} + \lambda_R +  \delta(\tau_0 - \Delta_g d) + \delta(\Delta_{\tilde{g}} d - \Delta_{g_0} d) + \delta^2 (|\nabla d|^2_{\tilde{g}} - 1)\big).
\end{align*}
Observe that the only term preventing us from applying the maximum principle is the last term in \eqref{the evolution equation of the scalar curvature formula}. We therefore estimate this term first. By the invariance of the Ricci curvature under diffeomorphisms $\Phi_t$, we obtain
\begin{equation}\label{Ricci curvature for NRDF}
    \begin{split}
        &\; \left| \tilde{\mathrm{Ric}}(t, x) - \frac{R_0}{n}\tilde{g}(x) \right|_{\tilde{g}(t, x)} = \left| \Phi_t^*(\hat{\mathrm{Ric}})(t, x) - \frac{R_0}{n}\Phi_t^*(\hat{g})(t, x)\right|_{\Phi_t^{*}(\hat{g})(t, x))}\\
        =&\; \left|\hat{\mathrm{Ric}}(t, \Phi^{-1}_t(x)) - \frac{R_0}{n}\hat{g}(t, \Phi^{-1}_t(x)) \right|_{\hat{g}(t, \Phi^{-1}_t(x))}.
    \end{split}
\end{equation}
By applying \Cref{long time existence with weight} with $\tau = \frac{\delta}{2}$, we obtain
\begin{equation}\label{Ricci curvature for NRDF 3}
    \begin{split}
        &\; \left| \hat{\mathrm{Ric}}(t, x) - \frac{R_0}{n} \hat{g}(t, x) \right|_{\hat{g}(t, x)} \\
        \leq& \; C(n, \varepsilon_0) \cdot \left| \hat{\mathrm{Ric}}(t, x) - \frac{R_0}{n} \hat{g}(t, x) \right| \\
        \leq&\; C(n, \varepsilon_0)\cdot \left( \left| \hat{\mathrm{Ric}}(t, x) - \frac{R_0}{n}g(x) \right| + \left| \frac{R_0}{n} \right| \cdot \left| \hat{g}(t, x) - g(x) \right| \right)\\
        \leq&\; C(n, \varepsilon_0)\cdot \left| \hat{g}^{-1}*\nabla^2\hat{g} + \hat{g}^{-1} * \hat{g}^{-1} *\nabla\hat{g} * \nabla \hat{g} \right| + \left| \frac{R_0}{n} \right| \cdot \left| \hat{g}(t, x) - g(x) \right| \\
        \leq&\; C(n, \varepsilon_0)\cdot \left( |\hat{g}^{-1}|\cdot |\nabla^2 \hat{g}| + |\hat{g}^{-1}| \cdot |\nabla \hat{g}|^2 \right) + \left| \frac{R_0}{n} \right| \cdot \left| \hat{g}(t, x) - g(x) \right| \\
        \leq&\; C(n,\varepsilon_0,\tau_0, k_0, \delta) \left( 1 + \frac{1}{\sqrt{t}} \right) e^{-\lambda t}\cdot e^{-\frac{\delta}{2} d(x)}.
    \end{split}
\end{equation}
Therefore, combining with \eqref{Ricci curvature for NRDF} and \eqref{Ricci curvature for NRDF 3}, we get 
\begin{equation}\label{evolution equation of scalar curvature formula 2}
    \begin{split}
        &\left| \hat{\mathrm{Ric}}(t, \Phi_t^{-1}(x)) - \frac{R_0}{n}\hat{g}(t, \Phi_t^{-1}(x)) \right|_{\hat{g}(t, \Phi_t^{-1}(x))}\\
        \leq&\; C(n,\varepsilon_0,\tau_0, k_0, \delta) \left(1 + \frac{1}{\sqrt{t}} \right) e^{-\lambda t} e^{-\frac{\delta}{2}d(\Phi_t^{-1}(x))}\\
        \leq&\; C(n,\varepsilon_0,\tau_0, k_0, \delta)\left(1 + \frac{1}{\sqrt{t}} \right) e^{-\lambda t} e^{-\frac{\delta}{2}d(x)} e^{\frac{\delta}{2}(d(x) - d(\Phi_t^{-1}(x)))}\\
        \leq&\; C(n,\varepsilon_0,\tau_0, k_0, \delta) \left(1 + \frac{1}{\sqrt{t}} \right) e^{-\lambda t} e^{-\frac{\delta}{2}d(x)} e^{\frac{\delta}{2}\cdot \mathrm{dist}(x, \Phi_t^{-1}(x))}\\
        \leq&\; C(n,\varepsilon_0,\tau_0, k_0, \delta) \left(1 + \frac{1}{\sqrt{t}} \right) e^{-\lambda t} e^{-\frac{\delta}{2}d(x)}.
    \end{split}
\end{equation}
The last inequality is by
\begin{align}\label{diffeo dist bounded}
    dist(x, \Phi_t^{-1}(x))\leq C(n,\varepsilon_0,\tau_0, k_0, \delta).
\end{align}
In fact, by \eqref{long time with weight result 1} of \Cref{long time existence with weight} with $\tau = \frac{\delta}{2}$, we have  
\begin{equation}
    \begin{split}
        |W(s, \Phi^{-1}_{s}(x))| =&\; |\hat{g}^{-1} * \hat{g}^{-1} * \nabla\hat{g}|\\
        \leq&  C(n) |\nabla(\hat{g}(t, \Phi_{s}^{-1}(x)) - g_0(\Phi_{s}^{-1}(x)))|\\
        \leq& C(n,\varepsilon_0,\tau_0, k_0, \delta) e^{-\lambda t}\cdot e^{- \frac{\delta}{2} d(\Phi_{s}^{-1}(x))}\\
        \leq& C(n,\varepsilon_0,\tau_0, k_0, \delta)e^{-\lambda t}.
    \end{split}
\end{equation}
Together with 
\begin{align}
    dist(x, \Phi_t^{-1}(x)) \leq \int_{0}^{t} |-W(s, \Phi_s^{-1}(x))| ds, 
\end{align}
\eqref{diffeo dist bounded} follows. Similarly, we can show that 
\begin{align}\label{evolution equation of scalar curvature formula 8}
    dist(x, \Phi_t(x))\leq C(n,\varepsilon_0,\tau_0, k_0, \delta).
\end{align}
Then, by \eqref{Ricci curvature for NRDF} and \eqref{evolution equation of scalar curvature formula 2}, we obtain
\begin{equation}\label{evolution equation of scalar curvature formula 3}
    \begin{split}
        |\tilde{\mathrm{Ric}}(t, x) - \frac{R_0}{n}\tilde{g}(x)|_{\tilde{g}(t, x)} \leq C(n,\varepsilon_0,\tau_0, k_0, \delta, T) e^{-\lambda t} e^{- \frac{\delta}{2} d(x)},\;\;\; \text{ for }\; (t, x)\in [T, +\infty) \times M.
    \end{split}
\end{equation}
In addition, for
\begin{align*}
    \delta \in \left(0, \frac{\tau_0}{2} + \sqrt{\frac{\tau_0^2}{4} - 2\frac{R_0}{n}} \right) \cap \left(\tau_0, \tau_0 + 2\sqrt{\frac{\tau_0^2}{4} - 2\gamma}\right), 
\end{align*}
we can always find $d_0 > 0$ sufficiently large and $\varepsilon_0, \lambda_R > 0$ sufficiently small such that for $(t, x)\in [T, +\infty)\times \Omega_{d_0}$,
\begin{align}\label{evolution equation of scalar curvature formula 5}
    A < 0 \quad \text{and} \quad \lambda_R < \lambda,
\end{align}
where $\Omega_{d_0}:= \{ x\in M \; | \; d(x) > d_0\}$. In fact, taking $t_0 = T$ in \Cref{long time existence without weights}, we may choose sufficiently small $\varepsilon_0 > 0$ and sufficiently large $d_0 > 0$ such that $|\Delta_{\tilde{g}}d - \Delta d|$ and $\big||\nabla d|_{\tilde{g}} - |\nabla d|\big|$ are sufficiently small for any $(t,x)\in[T, +\infty) \times \Omega_{d_0}$. Combining \eqref{the evolution equation of the scalar curvature formula}, \eqref{evolution equation of scalar curvature formula 3} and \eqref{evolution equation of scalar curvature formula 5}, we have  
\begin{equation}
    \begin{split}
        \frac{\partial}{\partial t} e^{\lambda_R t}e^{\delta d} (\tilde{R} - R_0) \leq&\; \Delta_{\tilde{g}} \left[ e^{\lambda_R t}e^{\delta d}(\tilde{R} - R_0) \right]- 2\delta \tilde{\nabla} d\cdot \tilde{\nabla} \left[ e^{\lambda_R t}e^{\delta d} (\tilde{R} - R_0) \right]\\ 
        &+ A\cdot e^{\lambda_R t}e^{\delta d}(\tilde{R} - R_0) + 2C(n, \varepsilon_0, \tau_0, k_0, \delta, T)e^{-\lambda_Rt}, 
    \end{split}
\end{equation}
on $[T, +\infty)\times \Omega_{d_0}$. Let 
\begin{equation}
    \begin{split}
        \varphi := e^{\lambda_R t} e^{\delta d}(\tilde{R} - R_0) + \frac{2C(n,\varepsilon_0,\tau_0, k_0, \delta, T)}{\lambda_R} e^{-\lambda_R t}.
    \end{split}
\end{equation}
Then, we obtain
\begin{align}\label{evolution equation of scalar curvature formula 6}
    \frac{\partial}{\partial t} \varphi \leq \Delta_{\tilde{g}}\varphi - 2\delta \tilde{\nabla} d\cdot \tilde{\nabla}\varphi + A\cdot\varphi,\quad \text{on }\; [T, +\infty)\times \Omega_{d_0}. 
\end{align}
By \eqref{The inequality of the scalar curvature}, \eqref{evolution equation of scalar curvature formula 5} and \eqref{evolution equation of scalar curvature formula 6}, we obtain
\begin{align}
    \frac{\partial}{\partial t} \varphi \leq \Delta_{\tilde{g}}\varphi - 2\delta \tilde{\nabla} d \cdot \tilde{\nabla} \varphi,\quad \text{on }\; [T, +\infty)\times \Omega_{d_0}.
\end{align}
By the maximum principle (Theorem 4.3 in \cite{QSW2013}), we have that 
\begin{align}\label{scalar curvature long time part 1}
    \sup_{(t, x)\in [T, +\infty)\times \Omega_{d_0}}\varphi(t, x) \leq \max\Big\{\sup_{x\in \Omega_{d_0}}\varphi(T, x), \sup_{(t, x)\in [T, +\infty)\times\partial\Omega_{d_0}}\varphi(t, x)\Big\}.
\end{align}
By \eqref{short time curvature} and \eqref{evolution equation of scalar curvature formula 8}, we have
\begin{align*}
    |\tilde{R}(t, x) - R_0| =&\; |\hat{R}(t, \Phi_t(x)) - R_0|\\
    \leq&\;  C(n, \|\hat{g}^{-1} - g^{-1}\|, \|\hat{g}^{-1}\|, \|\nabla\hat{g}\|, \|\mathrm{Rm}\|)\cdot (|\nabla\hat{g}| + |\nabla^2 \hat{g}| + |\hat{g}^{-1} - g^{-1}|)\\
    \leq&\; C(n, \varepsilon_0, k_0, \delta)(1 + \frac{1}{\sqrt{t}}) e^{-\delta d(\Phi_t(x))}\\
    \leq&\; C(n, \varepsilon_0, k_0, \delta) (1 + \frac{1}{\sqrt{t}}) e^{-\delta d(x)} e^{\delta\ \mathrm{dist}(x, \Phi_{t}(x))}\\
    \leq&\; C(n,\varepsilon_0,\tau_0, k_0, \delta) (1 + \frac{1}{\sqrt{t}}) e^{-\delta d(x)},\quad \text{for }\; (t, x)\in (0, T]\times M,
\end{align*}
which implies that 
\begin{align}\label{scalar curvature long time part 2}
    \sup_{x\in \Omega_{d_0}}\varphi(T, x) \leq C(n,\varepsilon_0,\tau_0, k_0, \delta, T).
\end{align}
Again, by \Cref{long time existence with weight} with $\tau = \frac{\delta}{2}$ and \eqref{diffeo dist bounded}, for any $(t, x)\in [T, +\infty)\times M$,
\begin{equation}\label{evolution equation of scalar curvature formula 9}
    \begin{split}
        |\tilde{R}(t, x) - R_0| =&\; |\hat{R}(t, \Phi_t(x)) - R_0|\\
        \leq&\;  C(n, \|\hat{g}^{-1} - g^{-1}\|, \|\hat{g}^{-1}\|, \|\nabla\hat{g}\|, \|\mathrm{Rm}\|)\cdot (|\nabla\hat{g}| + |\nabla^2 \hat{g}| + |\hat{g}^{-1} - g^{-1}|)\bigg|_{(t, \Phi_t(x))}\\
        \leq&\; C(n, \varepsilon_0, k_0, T) e^{-\lambda t} e^{-\frac{\delta}{2} d(\Phi_t(x))}\\
        \leq&\; C(n, \varepsilon_0, k_0, T) e^{-\lambda_R t} e^{-\frac{\delta}{2} d(x)}e^{\frac{\delta}{2}\mathrm{dist}(x, \Phi_t(x))}\\
        \leq&\; C(n,\varepsilon_0,\tau_0, k_0, \delta, T) e^{-\lambda_R t} e^{-\frac{\delta}{2} d(x)}, 
    \end{split}
\end{equation}
which implies that  
\begin{equation}\label{scalar curvature long time part 3}
    \begin{split}
        \sup_{(t, x)\in [T, +\infty)\times \partial\Omega_{d_0}} \varphi(t, x) \leq&\; \sup_{(t, x)\in [T, +\infty) \times \partial \Omega_{d_0}} e^{\lambda_R t} e^{\delta d(x)} |\tilde{R}(t, ,x)- R_0| + C(n, \varepsilon_0, \tau_0, k_0, \delta, T)\\
        \leq&\; C(n,\varepsilon_0,\tau_0, k_0, \delta, T) \sup_{(t, x)\in [T, +\infty) \times \partial\Omega_{d_0}} e^{\frac{\delta}{2} d(x)}\\
        \leq&\; C(n,\varepsilon_0,\tau_0, k_0, \delta, T).
    \end{split}
\end{equation}
Thus, by \eqref{scalar curvature long time part 1},\eqref{scalar curvature long time part 2} and \eqref{scalar curvature long time part 3}, we derive 
\begin{align}\label{evolution equation of scalar curvature formula 10}
    \sup_{(t, x)\in [T, +\infty)\times \Omega_{d_0}}|\tilde{R}(t, x) - R_0| \leq C(n,\varepsilon_0,\tau_0, k_0, \delta, T) e^{-\lambda_R t} e^{- \delta d(x)}.
\end{align}
Moreover, by \eqref{evolution equation of scalar curvature formula 9}, 
\begin{equation}\label{evolution equation of scalar curvature formula 11}
    \begin{split}
        \sup_{(t, x)\in [T, +\infty) \times (M\setminus \Omega_{d_0})} e^{\lambda_R t} e^{\delta d(x)}|\tilde{R}(t, x) - R_0| \leq&\; C(n,\varepsilon_0,\tau_0, k_0, \delta, T)\sup_{(t, x)\in [T, +\infty)\times (M\setminus\Omega_{d_0})}e^{\frac{\delta}{2}d(x)}\\
        \leq&\; C(n,\varepsilon_0,\tau_0, k_0, \delta, T).
    \end{split}
\end{equation}
Combining \eqref{evolution equation of scalar curvature formula 10} and \eqref{evolution equation of scalar curvature formula 11}, we obtain 
\begin{equation}\label{evolution equation of scalar curvature formula 12}
    \begin{split}
        \sup_{(t, x)\in [T, +\infty)\times M}|\tilde{R}(t, x) - R_0| \leq C(n,\varepsilon_0,\tau_0, k_0, \delta) e^{-\lambda_R t} e^{- \delta d(x)}.
    \end{split}
\end{equation}
By \eqref{relations of the scalar curvature of NRF and NRDF}, \eqref{diffeo dist bounded} and \eqref{evolution equation of scalar curvature formula 12}, we have that for $(t, x)\in [T, +\infty) \times M$,
\begin{equation}\label{long time curvature}
    \begin{split}
        |\hat{R}(t, x) - R_0| = |\tilde{R}(t, \Phi_t^{-1}(x)) - R_0| \leq&\; C(n,\varepsilon_0,\tau_0, k_0, \delta) e^{-\lambda_R t} e^{- \delta d(\Phi_t^{-1}(x))} \\
        \leq&\; C(n,\varepsilon_0,\tau_0, k_0, \delta) e^{-\lambda_R t} e^{- \delta d(x)} e^{\delta\ \mathrm{dist}(x, \Phi^{-1}_s(x))}\\
        \leq&\; C(n,\varepsilon_0,\tau_0, k_0, \delta) e^{-\lambda_R t} e^{- \delta d(x)}. 
    \end{split}
\end{equation}
Finally, \eqref{scalar curvature estimate with weight} follows from \eqref{short time curvature} in step 2 and \eqref{long time curvature}. 
\end{proof}
Note that the admissible range of $\delta$ differs between \Cref{The evolution equation of the deTurck term} and \Cref{The evolution equation of the scalar curvature}. Since $R_0<0$ in our setting, we take
\begin{align}\label{final region of delta}
    \delta \in \left( 0, \frac{\tau_0}{2} + \sqrt{\frac{\tau_0^2}{4} - \frac{R_0}{n}} \right) \cap \left( \tau_0, \tau_0 + 2\sqrt{\frac{\tau_0^2}{4} - 2\gamma} \right)
\end{align}
throughout the subsequent sections. 
\section{Relative Volume Comparison and Rigidity}\label{sec: relative_volume}

In this section, we analyze the behavior of the relative volume $V_{g_0}(\hat{g}(t))$ along the NRDF \eqref{NRDF} under the assumption that $R \geq R_0$, where the background manifold, $(M^n, g_0)$, is a simply connected irreducible symmetric space of non-compact type of rank 1 and $R$ is the scalar curvature of a metric $g$ which is the initial metric of the NRDF \eqref{NRDF}, satisfying the perturbation condition $\|e^{\delta d}(g - g_0)\|_{C^1} < \varepsilon_0$ for sufficiently small $\varepsilon_0 > 0$ and $\delta$ satisfying \eqref{final region of delta}. Under these constraints, we establish \Cref{main theorem}. To avoid cumbersome notation, we use $C$ to denote generic positive constants that may vary from line to line but are independent of time and space, unless otherwise stated. For convenience, define
\begin{equation}\label{definition of partial relative volume}
    V_{g_0}(\Omega, \hat{g}) := \mathrm{Vol}(\Omega, \hat{g}) - \mathrm{Vol}(\Omega, g_0),
\end{equation}
for some bounded region $\Omega \subset M$. 

\textbf{Proof of \Cref{main theorem}: } We divide the proof into five steps. 

\textbf{Step 1:} For $(M, g_0)$ a simply connected irreducible symmetric space of non-compact type of rank~1, by \Cref{symmetric spaces of non-compact type of rank 1} we see that, under the dimension requirements of \Cref{main theorem}, we have
\begin{align}\label{key intersection set 1}
    \left(0, \frac{\tau_0}{2} + \sqrt{\frac{\tau_0^2}{4} - \frac{R_0}{n}} \right) \cap \left( \tau_0, \tau_0 + 2\sqrt{\frac{\tau_0^2}{4} - 2\gamma} \right)\neq \emptyset. 
\end{align}
In fact, it suffices to check if $\frac{\tau_0^2}{4} > 2\gamma$. Moreover, $(M, g_0)$ is a Cartan-Hadamard Einstein manifold and strictly stable (Section 5.3 in \cite{Ba2015}). Therefore, there is no obstacle for us to apply \Cref{short time existence with weight,long time existence with short time weight,long time existence with weight,The evolution equation of the deTurck term,The evolution equation of the scalar curvature} in the previous sections. 

\textbf{Step 2: } In this step, we shall show that $V_{g_0}(\hat g(t))$ exists for all $t\in[0,+\infty)$ and is independent of the choice of compact exhaustion. This follows from \Cref{existence of relative volume} and \Cref{long time existence with short time weight}. Specifically, for any $\delta>\tau_0$ and $T>0$, \Cref{long time existence with short time weight} yields
\begin{equation}\label{final proof formula 0}
\|e^{\delta d(\cdot)}\bigl(\hat{g}(t, \cdot)-g_0(\cdot)\bigr)\|_{C^1}
\le C(n,\varepsilon_0,\delta,\tau_0,k_0,T),
\quad \text{for }\; t\in[0,T].
\end{equation}
Combining with \Cref{existence of relative volume} completes the proof of the existence of $V_{g_0}(\hat{g}(t))$.

\textbf{Step 3: } In this step, we shall show 
\begin{equation}\label{non-increasing of relative volume}
    \frac{d}{dt}V_{g_0}(\hat{g}(t)) \leq 0,\quad \text{for }\; t\in (0, +\infty). 
\end{equation}
It suffices to show that 
\begin{equation}
    \begin{split}
        \frac{d}{dt}V_{g_0}(\hat{g}(t)) = -\int_{M} (R(\hat{g}(t)) - R_0)) d\mu_{\hat{g}(t)},\quad \text{and}\quad R(\hat{g}(t)) \geq R_0,
    \end{split}
\end{equation}
for any $t\in (0, +\infty)$. We first write down the evolution equation of the relative volume under the NRDF \eqref{NRDF} on $B(p_0, s)$ which is a natural compact exhaustion on a simply connected symmetric space of non-compact type of rank 1. In addition, in our setting, there is a global normal coordinate at a fixed point $p_0\in M$. That is $M\setminus \{p_0\} \cong \mathbb{R}_+ \times \mathbb{S}^{n - 1}$ with $g_0 = dr^2 + g_r$, where $g_r$ is a family of metrics on $\mathbb{S}^{n - 1}$. Moreover, 
\begin{align*}
    d\mu_{\hat{g}} =\; \sqrt{\mathrm{det}(\hat{g})}\ dr d\theta\;\;\; \text{and}\;\;\;
    d\mu_{g_0} = \sqrt{\mathrm{det}(g_0)}\ drd\theta = \sqrt{\mathrm{det}(g_r)}\ dr d\theta, 
\end{align*}
where $\theta\in \mathbb{S}^{n - 1}$, $d\theta$ is the volume form induced by the standard metric of unit sphere $\mathbb{S}^{n - 1}$, and $d\mu_{\hat{g}}$ and $d\mu_{g_0}$ are the volume forms for $\hat{g}$ and $g_0$ respectively. Thus, 
\begin{equation}\label{final proof formula 1}
    \begin{split}
        &\frac{d}{dt}V_{g_0}(B(p_0, s), \hat{g}) = \frac{d}{dt}(\mathrm{Vol}(B(p_0, s), \hat{g}) - \mathrm{Vol}(B(p_0, s), g_0))\\
        =&\; \frac{d}{dt}\int_{B(p_0, s)} d\mu_{\hat{g}} - d\mu_{g_0} = \frac{d}{dt}\int_{B(p_0, s)} \sqrt{\det(\hat{g})} - \sqrt{\det(g_0)}\;dr d\theta\\
        =&\; -\int_{B(p_0, s)}(R(\hat{g}) - R_0)\sqrt{\mathrm{det}(\hat{g})} \ drd\theta+\int_{B(p_0, s)} \hat{g}^{ij}\hat{\nabla}_{j} W_{i}\cdot \sqrt{\mathrm{det}(\hat{g})} drd\theta\\
        =&\; -\int_{B(p_0, s)}(R(\hat{g}) - R_0)\ d\mu_{\hat{g}} + \int_{B(p_0, s)} \hat{g}^{ij}\hat{\nabla}_{j} W_{i}\ d\mu_{\hat{g}}\\
        =&\; -\int_{B(p_0, s)}(R(\hat{g}) - R_0) \ d\mu_{\hat{g}} + \int_{\partial B(p_0, s)} \hat{g}^{ij} W_{i} v_{j} d\sigma_{\hat{g}},
    \end{split}
\end{equation}
where $W_i = \hat{g}_{ik} \hat{g}^{pq}\bigl(\Gamma_{pq}^k(\hat{g}(t)) - \Gamma_{pq}^k\bigr)$, $v$ denotes the unit outward normal to $\partial B(p_0, s)$ with respect to the metric induced by $\hat{g}$, and $d\sigma_{\hat{g}}$ is the corresponding induced volume form on $\partial B(p_0, s)$ (Remark: $B(p_0, s)$ is a geodesic ball with respect to $g_0$). We shall prove the limit of the above exists as $s \rightarrow +\infty$. By \eqref{sphere area blow up}, we obtain 
\begin{align}\label{area blow up}
    \sqrt{\mathrm{det}(g_r)} \leq C e^{\tau_1 r}, \quad\text{with }\; \tau_1 = \frac{\delta + \tau_0}{2}. 
\end{align}
By \Cref{long time existence with weight} with $\tau = \frac{\delta}{2}$, we obtain 
\begin{equation}\label{metric decay condition}
    |\hat{g}(t, x) - g(x)| \leq Ce^{-\lambda_R t}e^{-\frac{\delta}{2}d(x)},
\end{equation}
for some constant $\lambda_R > 0$ and $t \geq  0$. In addition, by \Cref{The evolution equation of the deTurck term} and \Cref{The evolution equation of the scalar curvature}, for $\delta$ satisfying \eqref{final region of delta} (in particular, $\delta > \tau_0$), we obtain
\begin{equation}\label{three conditions for main theorem}
    \begin{split}
        (1)&\; R(\hat{g})(t, x) \geq R_0;\\
        (2)&\; |R(\hat{g})(t, x) - R_0(x)| \leq C\cdot (1 + \frac{1}{\sqrt{t}}) \cdot e^{-\lambda_R t} \cdot e^{-\delta d(x)};\\
        (3)&\; |W(t, x)| \leq C\cdot e^{-\lambda_R t} e^{-\delta d(x)},
    \end{split}
\end{equation}
for $t > 0$. By \eqref{area blow up}, \eqref{metric decay condition} and condition (2) in \eqref{three conditions for main theorem}, we obtain 
\begin{equation}\label{curvature decay precursor}
    \begin{split}
        |R(\hat{g}) - R_0|\cdot \sqrt{\mathrm{det}(\hat{g})} \leq&\; |R(\hat{g}) - R_0|\sqrt{\mathrm{det}(g_r)} + |R(\hat{g}) - R_0|\cdot |\frac{\sqrt{\mathrm{det}(\hat{g})}}{\sqrt{\mathrm{det}(g_0)}} - 1|\cdot \sqrt{\mathrm{det}(g_r)}\\
        \leq&\; |R(\hat{g}) - R_0|\sqrt{\mathrm{det}(g_r)} + |R(\hat{g}) - R_0|\cdot |\hat{g} - g_0|\cdot \sqrt{\mathrm{det}(g_r)}\\
        \leq&\; C\cdot (1 + \frac{1}{\sqrt{t}}) e^{-\lambda_R t} e^{-(\delta - \tau_1)r} + C(1 + \frac{1}{\sqrt{t}}) e^{-2\lambda_R t} e^{-(\frac{3}{2}\delta - \tau_1)r} \\
        \leq&\; C\cdot (1 + \frac{1}{\sqrt{t}}) e^{-\lambda_R t} e^{-(\delta - \tau_1)r}. 
    \end{split}
\end{equation}
Thus, 
\begin{equation}\label{curvature decay}
    \begin{split}
        &\int_{B(p_0, s)}|R(\hat{g}) - R_0|d\mu_{\hat{g}} =\int_{0}^{s}\int_{\mathbb{S}^{n - 1}}|R(\hat{g}) - R_0|\sqrt{\mathrm{det}(\hat{g})}\ d\theta dr\\
        \leq&\;  C\int_{0}^{s}\int_{\mathbb{S}^{n - 1}} \left(1 + \frac{1}{\sqrt{t}} \right) e^{-\lambda_R t} e^{-(\delta - \tau_1)r} d\theta dr\\
        \leq&\; C \cdot \left(1 + \frac{1}{\sqrt{t}} \right) e^{- \lambda_R t} (1 - e^{- (\delta - \tau_1) s}), 
    \end{split}
\end{equation}
which implies that 
\begin{align}\label{curvature converge}
    \int_{M}(R(\hat{g}) - R_0)d\mu_{\hat{g}}\;\;\;  \text{is absolutely integrable, for any fixed }\; t\in (0, +\infty).   
\end{align}
Thus, 
\begin{align}
    \lim_{s\rightarrow +\infty} \int_{B(p_0, s)} (R(\hat{g}) - R_0)d\mu_{g_0} = \int_{M} (R(\hat{g}) - R_0) d\mu_{g_0},\quad \text{for }\; t\in (0, +\infty). 
\end{align}
By \eqref{area blow up}, \eqref{metric decay condition} and condition (3) in \eqref{three conditions for main theorem}, we obtain
\begin{equation}\label{DeTurck term decay process}
    \begin{split}
        |\int_{\partial B(p_0, s)}\hat{g}^{ij}& W_i v_j\ d\sigma_{\hat{g}}|\leq\; \int_{\partial B(p_0, s)}|\hat{g}^{ij} W_i v_j|\ d\sigma_{\hat{g}}\\
        \leq&\; C\int_{\partial B(p_0, s)} |W|\ d\sigma_{g_{0}}\\
        \leq&\; C\int_{\mathbb{S}^{n - 1}}|W|\cdot\sqrt{\mathrm{det}(g_s)}\ d\theta \\
        \leq&\; Ce^{-\lambda_R t}e^{-(\delta - \tau_1)s},
    \end{split}
\end{equation}
which implies that 
\begin{align}\label{DeTurck term converge}
    \lim_{s\rightarrow 0}\int_{\partial B(p_0, s)}\hat{g}^{ij} W_i v_j\ d\sigma_{\hat{g}} = 0\quad \text{for }\; t\in [0, +\infty). 
\end{align}
By \eqref{final proof formula 1}, \eqref{curvature converge} and \eqref{DeTurck term converge}, we derive 
\begin{equation}\label{outer limit}
    \begin{split}
        \lim_{s\rightarrow +\infty}\frac{d}{dt} V_{g_0}(B(p_0, s), \hat{g}(t)) = - \int_{M} (R(\hat{g}(t)) - R_0) d\mu_{\hat{g}}. 
    \end{split}
\end{equation} 
The above convergence is uniform with respect to $t\in [t_0, +\infty)$ for any $t_0 > 0$. On the other hand, by \eqref{final proof formula 0} and \eqref{area blow up}, 
\begin{equation}\label{volume converge precursor}
    \begin{split}
        \Big|(\frac{\sqrt{\mathrm{det}(\hat{g})}}{\sqrt{\mathrm{det}(g_0)}} - 1) \cdot\sqrt{\mathrm{det}(g_r)}\Big| \leq&\; |\hat{g} - g_0| \cdot \sqrt{\mathrm{det}(g_r)} \\
        \leq&\; C(n,\varepsilon_0,\tau_0, k_0, \delta, T) e^{-(\delta - \tau_1)r},\quad \text{for }\; t\in [0, T]. 
    \end{split}
\end{equation}
Thus, 
\begin{equation}
    \begin{split}
        &\Big|V_{g_0}(B(p_0, s_2), \hat{g}(t)) - V_{g_0}(B(p_0, s_1), \hat{g}(t))\Big|\\
        =& \Big|\int_{B(p_0, s_2)\setminus B(p_0, s_1)} d\mu_{\hat{g}} - d\mu_{g_0}\Big| = \Big|\int_{s_1}^{s_2}\int_{\mathbb{S}^{n - 1}} (\frac{\sqrt{\mathrm{det}(\hat{g})}}{\sqrt{\mathrm{det}(g_0)}} - 1) \cdot\sqrt{\mathrm{det}(g_r)}\ d\theta dr \Big|\\
        \leq&\; C(n,\varepsilon_0,\tau_0, k_0, \delta, T) \int_{s_1}^{s_2}e^{- (\delta - \tau_1) r}\  dr\\
        \leq&\; C(n,\varepsilon_0,\tau_0, k_0, \delta, T)(e^{-(\delta - \tau_1)s_1} - e^{-(\delta - \tau_1)s_2}),
    \end{split}. 
\end{equation}
for $t\in [0, T]$, which implies 
\begin{align}\label{inner limit}
    \lim_{s \rightarrow +\infty} V_{g_0}(B(p_0, s), \hat{g}) = V_{g_0}(\hat{g})
\end{align}
is uniform convergence with respect to $t\in [0, T]$, for any $T > 0$. Combining with \eqref{outer limit} and \eqref{inner limit}, we obtain 
\begin{equation}\label{final proof formula 2}
    \begin{split}
        \frac{d}{dt}V_{g_0}(\hat{g}(t)) =&\; 
        \frac{d}{dt}\lim_{s\rightarrow +\infty}V_{g_0}(B(p_0, s), \hat{g})= \lim_{s\rightarrow +\infty} \frac{d}{dt} V_{g_0}(B(p_0, s), \hat{g})\\ 
        =& -\int_{M}(R(\hat{g}) - R_0)d\mu_{\hat{g}}, 
    \end{split}
\end{equation} 
with respect to $t\in [t_0, T]$, for any $T > t_0 > 0$. Therefore, by \eqref{final proof formula 2} and the condition (3) in \eqref{three conditions for main theorem}, we obtain \eqref{non-increasing of relative volume}.

\textbf{Step 4: } In this step, we shall show that 
\begin{equation}\label{limit of relative volume}
    \lim_{t\rightarrow +\infty} V_{g_0}(\hat{g}(t)) = 0.
\end{equation}
By \eqref{final proof formula 1}, \eqref{curvature decay precursor}, \eqref{DeTurck term decay process}, \eqref{volume converge precursor} and \eqref{final proof formula 2}, for any small $\eta > 0$, we can always find $s_3 > 0$ such that 
\begin{equation}\label{limit of relative volume formula 1}
    \begin{split}
        &\Big|V_{g_0}(M\setminus B(p_0, s_3), \hat{g}) \Big| \leq \Big|V_{g_0}(M\setminus B(p_0, s_3), \hat{g}) - V_{g_0}(M \setminus B(p_0, s_3), g)\Big| + \Big| V_{g_0}(M \setminus B(p_0, s_3), g)\Big|\\
        =&\; \big|-\int_{0}^{t}\int_{M\setminus B(p_0, s_3)}(R(\hat{g}) - R_0)\ d\mu_{\hat{g}} dw - \int_{0}^{t} \int_{\partial B(p_0, s_3)} \hat{g}^{ij}W_i v_j\ d\sigma_{\hat{g}}dw \Big|+ \Big|\int_{M \setminus B(p_0, s_3)} d\mu_{\hat{g}} - d\mu_{g_0}\big|\\
        \leq&\;\int_{0}^{t}\int_{s_3}^{+\infty}\int_{\mathbb{S}^{n - 1}}|R(\hat{g}) - R_0|\sqrt{\mathrm{det}(\hat{g})}\ d\theta dr dw + \int_{0}^{t} \int_{\partial B(p_0, s_3)} |\hat{g}^{ij}W_i v_j|\ d\sigma_{\hat{g}}dw\\ 
        &+ \int_{s_3}^{+\infty}\int_{\mathbb{S}^{n - 1}}|\frac{\sqrt{\mathrm{det}(g)}}{\sqrt{\mathrm{det}(g_0)}} - 1|\cdot \sqrt{\mathrm{det}(g_r)}\ d\theta dr\\
        \leq&\; C\int_{0}^{t}\int_{s_3}^{+\infty}(1 + \frac{1}{\sqrt{w}}) e^{-\lambda_R w} e^{-(\delta - \tau_1)r}\ drdw + C\int_{0}^{t}e^{-\lambda_R w}e^{-(\delta - \tau_1)s_3}\ dw + C\int_{s_3}^{+\infty} e^{-(\delta - \tau_1)r} dr\\
        \leq& Ce^{-(\delta - \tau_1) s_3} < \frac{\eta}{2}.  
    \end{split}
\end{equation}
Fix the above $s_3$. By \eqref{final proof formula 1}, \eqref{curvature decay} and \eqref{DeTurck term decay process} and \Cref{long time existence with weight} with $\tau = \frac{\delta}{2}$, we can always find $T_1 > 0$ such that for any $t > T_1$
\begin{equation}\label{limit of relative volume formula 2}
    \begin{split}
        &\Big| V_{g_0}(B(p_0, s_3), \hat{g}(t))\Big| \leq\Big| V_{g_0}(B(p_0, s_3), \hat{g}(t)) - V_{g_0}(B(p_0, s_3), \hat{g}(T_1)) \Big| + \Big|V_{g_0}(B(p_0, s_3), \hat{g}(T_1)) \Big|\\
        \leq&\; \int_{T_1}^{t}\int_{0}^{s_3}\int_{\mathbb{S}^{n - 1}}|R(\hat{g}) - R_0|\sqrt{\mathrm{det}(\hat{g})}\ d\theta dr dw + \int_{T_1}^{t} \int_{\partial B(p_0, s_3)} |\hat{g}^{ij}W_i v_j|\ d\sigma_{\hat{g}}dw\\
        &+ C\int_{0}^{s_3}|\hat{g}(T_1, .) - g_0(\cdot)|\sqrt{\mathrm{det}(g_r)} dr\\
        \leq&\;C\int_{T_1}^{t}\int_{0}^{s_3}(1 + \frac{1}{\sqrt{w}}) e^{-\lambda_R w} e^{-(\delta - \tau_1)r}\ drdw + C\int_{T_1}^{t}e^{-\lambda_R w}e^{-(\delta - \tau_1)s_3}\ dw + C\int_{0}^{s_3} e^{-\lambda_R T_1}e^{-(\frac{\delta}{2} - \tau_1)r} dr\\
        \leq&\; C e^{(\tau_1 - \frac{\delta}{2})s_3}(1 + \frac{1}{\sqrt{T_1}})e^{-\lambda_R T_1} < \frac{\eta}{2}.
    \end{split}
\end{equation}
Combining \eqref{limit of relative volume formula 1} and \eqref{limit of relative volume formula 2}, we obtain \eqref{limit of relative volume}. Finally, combining \eqref{non-increasing of relative volume} and \eqref{limit of relative volume}, we obtain
\begin{align}\label{positive volume}
    V_{g_0}(g) = V_{g_0}(\hat{g}(0)) \geq V_{g_0}(\hat{g}(t)) \geq \lim_{t\rightarrow +\infty} V_{g_0}(\hat{g}(t)) = 0. 
\end{align}

\textbf{Step 5: } In this step, we shall show that if
\begin{align*}
    V_{g_0}(g) = 0,
\end{align*}
there is a diffeomorphism $\Phi: M \rightarrow M$ such that 
\begin{align*}
    \Phi^*(g) = g_0. 
\end{align*}
By \eqref{positive volume}, we obtain that for any $t\in (0, +\infty)$, 
\begin{equation*}
    0 = V_{g_0}(g) = V_{g_0}(\hat{g}(0)) \geq V_{g_0}(\hat{g}(t)) \geq \lim_{t\rightarrow +\infty} V_{g_0}(\hat{g}(t)) = 0,
\end{equation*}
which implies 
\begin{align*}
    V_{g_0}(\hat{g}(t)) \equiv 0,\quad \text{for any }\; t\in [0, +\infty).  
\end{align*}
By \eqref{final proof formula 2}, we obtain  
\begin{equation*}
    \int_{M} (R(\hat{g}) - R_0) d\mu_{\hat{g}} = -\frac{d}{dt}V_{g_0}(\hat{g}(t)) \equiv 0.
\end{equation*}
Combining with the condition (1) in \eqref{three conditions for main theorem}, we obtain 
\begin{align*}
    R(\hat{g})(t, x) \equiv R_0,\quad \text{for any }\; (t, x)\in [0, +\infty) \times M,  
\end{align*}
which implies that 
\begin{align*}
    R(\tilde{g})(t, x) \equiv R_0,\quad \text{for any }\; (t, x)\in [0, +\infty) \times M. 
\end{align*}
where $\tilde{g}(t, x)$ is the corresponding solution to the NRF \eqref{NRF}, with initial data $g(x)$. By the evolution equation of the scalar curvature \eqref{evolution equation of scalar curvature without weight} of $\tilde{g}(t, x)$, we derive 
\begin{align*}
    \mathrm{Ric}[\tilde{g}](t, x) \equiv \frac{R_0}{n}\tilde{g}(t, x),\quad \text{for any }\; (t, x)\in [0, +\infty) \times M, 
\end{align*}
which implies that
\begin{align*}
    \tilde{g}(t)\equiv \tilde{g}(0) = g.
\end{align*}
Combining with \eqref{gauge term diffeomorphism}, we obtain 
\begin{align}\label{nrdf and nrf for still case}
    \Phi_t^*(\hat{g}(t)) = \tilde{g}(t) \equiv g. 
\end{align}
On the other hand, by \eqref{gauge term diffeomorphism} and \Cref{The evolution equation of the deTurck term}, there is a diffeomorphism $\Phi : M \rightarrow M$ such that 
\begin{align}\label{convergence of gauge diffeomorphism}
    \lim_{t\rightarrow +\infty}\Phi_t^* =\Phi\quad \text{in }\; C^{\infty}(M, M). 
\end{align}
By \Cref{long time existence without weights}, \eqref{nrdf and nrf for still case} and \eqref{convergence of gauge diffeomorphism}, we obtain
\begin{align}
    \Phi^*(g_0) = \lim_{t\rightarrow +\infty}\Phi_t^*(\hat{g}(t)) \equiv g.
\end{align}

\section*{Acknowledgement}
We would like to thank Professor Jie Qing, the advisor of the first two authors; Professor Wei Yuan, the advisor of the third author; and Professor Fang Wang for their continuous support. Additionally, we are grateful to Dr. Xu Gao and Dr. Yiyi Zhu for their valuable insights and numerous helpful discussions.

\printbibliography
\end{document}